\documentclass[preprint,12pt]{elsarticle}
\usepackage{enumerate,amsmath,amsthm,mathrsfs,array,subfigure,graphicx}
\usepackage{a4}
\usepackage{xcolor}
\usepackage{amsfonts}
\usepackage{amsthm}
\usepackage{amssymb}
\usepackage{graphicx}
\usepackage{tabularx}
\usepackage{amsmath}
\usepackage{latexsym}
\usepackage{setspace}
\usepackage{longtable}
\usepackage{url}
\usepackage{caption}
\usepackage[margin=2.5cm]{geometry}

\usepackage{tikz}
\usepackage{tikz,pgfplots}
\usetikzlibrary{decorations.markings}

\usepackage{color}
\usepackage{graphicx}
\usepackage{caption} 

\newcommand{\g}{{\rm g}}
\DeclareMathOperator{\cp}{\,\square\,}

\newtheorem{thm}{Theorem}

\newtheorem{lem}[thm]{Lemma}

\newtheorem{prb}[thm]{Problem}
\newtheorem{obs}[thm]{Observation}
\newtheorem{coro}[thm]{Corollary}
\newtheorem{dfn}[thm]{Definition}

\renewcommand{\deg}{{\rm deg}}
\newcommand{\gp}{{\rm gp}}
\newcommand{\ip}{{\rm ip}}

\footskip=30pt \vspace{5cm}

\numberwithin{figure}{section}

\title{On the General Position Number of Mycielskian Graphs}\date{}
\journal{Discrete Applied Mathematics}

\begin{document}
	
	\newcommand{\TODO}[1]{\textcolor{red}{TODO: #1}}

	\begin{frontmatter}
		


		
		\author[label5,label1]{Elias John Thomas}
		\ead{eliasjohnkalarickal@gmail.com}
		\author[label2]{Ullas Chandran S. V.}
		\ead{svuc.math@gmail.com}
		\author[label3]{James Tuite}
		\ead{james.t.tuite@open.ac.uk}
		\author[label4]{Gabriele {Di Stefano}}
		\ead{gabriele.distefano@univaq.it}
		
		\address[label5]{Department of Mathematics, Mar Ivanios College, University of Kerala, Thiruvananthapuram-695015, Kerala, India}
		
		\address[label1]{Department of Mathematics, Greenshaw High School, Grennell Road, Sutton, UK}
		
		\address[label2]{Department of Mathematics, Mahatma Gandhi College, University of Kerala, Thiruvananthapuram-695004, Kerala, India}
		
		\address[label3]{School of Mathematics and Statistics, Open University, Walton Hall, Milton Keynes, UK}	
		
		\address[label4]{Department of Information Engineering, Computer Science and Mathematics, University of L'Aquila, Italy}
		
		\begin{abstract}
			The general position problem for graphs was inspired by the no-three-in-line problem from discrete geometry. A set $S$ of vertices of a graph $G$ is a \emph{general position set} if no shortest path in $G$ contains three or more vertices of $S$. The \emph{general position number} of $G$ is the number of vertices in a largest general position set. In this paper we investigate the general position numbers of the Mycielskian of graphs. We give tight upper and lower bounds on the general position number of the Mycielskian of a graph $G$ and investigate the structure of the graphs meeting these bounds. We determine this number exactly for common classes of graphs, including cubic graphs and a wide range of trees. 
		\end{abstract}
		
		
	\end{frontmatter}

	\section{Introduction}\label{sec:intro}
	
	A \emph{graph} $G$ consists of a set of \emph{vertices} $V(G)$ connected by \emph{edges}, which are unordered pairs of vertices. All graphs used here are simple, connected and finite. We write $u \sim v$ if the vertices $u,v \in V(G)$ are adjacent. A \emph{path} $P_{\ell +1}$ of length $\ell $ in $G$ is a sequence $v_0,v_1,\dots ,v_{\ell }$ of distinct vertices such that $v_iv_{i+1}$ is an edge for $0 \leq i \leq \ell -1$. If $P$ is a path $v_0,v_1,\dots ,v_{\ell }$, then the \emph{reverse path} $\widetilde{P}$ of $P$ is the path $v_{\ell },\dots ,v_1,v_0$. The \emph{distance} $d_G(u,v)$ between vertices $u$ and $v$ of $G$ is the length of a shortest path from $u$ to $v$. A shortest path from $u$ to $v$ is also called a \emph{$u,v$-geodesic}. The \emph{neighbourhood} $N_G(u)$ of $u \in V(G)$ is the set $\{ v \in V(G): v \sim u\} $ of all vertices adjacent to $u$. More generally, for $t \geq 0$, $N^t(u)$ is the set of all vertices at distance $t$ from $u$, so that $N_G(u) = N^1(u)$. The \emph{degree} $\deg (u)$ is the number $|N_G(u)|$ of neighbours of $u$. A \emph{leaf} of $G$ is a vertex with degree one in $G$ and we define the \emph{leaf number} $\ell (G)$ to be the number of leaves in $G$. A vertex adjacent to a leaf is a \emph{support vertex}. A vertex of degree $n-1$ in a graph with order $n$ is \emph{universal}.
	
	A cycle $C_{\ell }$ of length $\ell $ is a sequence of vertices $v_0,v_1,\dots ,v_{\ell -1}$ such that $v_i\sim v_{i+1}$ for $0 \leq i \leq \ell -2$ and also $v_0 \sim v_{\ell -1}$. A connected graph with no cycles is a \emph{tree}. The \emph{girth} $\g(G)$ of $G$ is the length of a shortest cycle in $G$ (if $G$ is acyclic, then we adopt the convention that $\g(G) = \infty $). For $U \subseteq V(G)$ we write the subgraph induced by $U$ as $\langle U\rangle $. An independent set in $G$ is a set of mutually non-adjacent vertices, i.e. a set of vertices that induces an empty subgraph of $G$, and the cardinality $\alpha (G)$ of a largest independent set in $G$ is the \emph{independence number} of $G$. By contrast, a \emph{clique} in a graph is a set of mutually adjacent vertices. The complete graph $K_n$ is the graph with order $n$ such that every pair of vertices is adjacent, i.e. the whole vertex set is a clique.  A \emph{matching} $M$ in $G$ is an independent set of edges, i.e. if $e_1,e_2$ are two distinct edges of $M$, then $e_1$ and $e_2$ do not have an endvertex in common. The size of the largest matching in $G$ is the \emph{matching number} of $G$ and is denoted by $\nu(G)$. For any subsets $V_1,V_2 \subseteq V(G)$ we denote the set of edges of $G$ from $V_1$ to $V_2$ in $G$ by $(V_1,V_2)$. For any graph-theoretical terminology not defined here we take~\cite{BonMur} as our standard. 
	
	The \emph{general position problem} originated in Dudeney's no-three-in-line problem and the general position subset selection problem from discrete geometry~\cite{dudeney1958amusements,froese2017finding,payne2013general}. The general position problem was generalised to graph theory independently in~\cite{ullas-2016} and~\cite{manuel2018general} as follows. A set $S$ of vertices in a graph $G$ is a \emph{general position set}, or is in \emph{general position}, if no shortest path in $G$ contains three or more vertices of $G$. A  largest general position set of $G$ is called a \emph{gp-set} and its cardinality $\gp(G)$ is the \emph{general position number} of $G$ (or \emph{gp-number} for short). The structure of general position sets was characterised in~\cite{anand2019characterization}. In particular, it was shown in this article that any general position set induces an \emph{independent union of cliques}, i.e. a disjoint union of one or more cliques $W_1,\dots ,W_t$, $t \geq 1$, with no edges between cliques $W_i$ and $W_j$ if $1 \leq i < j \leq t$. These cliques must also have the \emph{distance-constant} property, meaning that for any cliques $W_i$ and $W_j$ in this collection and any vertices $w_1,w_1^\prime \in W_i$, $w_2,w_2^\prime \in W_j$, we have $d_G(w_1,w_2) = d_G(w_1^\prime ,w_2^\prime )$.
	
	We make use of two variants of general position sets. The largest number of vertices in a general position set that is also an independent set is called the \emph{independent position number} $\ip (G)$ and was studied in~\cite{independentposition}. The article~\cite{KlaRalYer} defined a \emph{general $d$-position set} to be a subset $S \subseteq V(G)$ such that no shortest path of length at most $d$ contains three or more vertices of $S$. The largest number of vertices in a general $d$-position set is denoted by $\gp _d(G)$.  Other types of position sets have also been investigated, including \emph{monophonic position sets}~\cite{ThomasChandranTuite}, \emph{mutual visibility sets}~\cite{mutualvisibility}, \emph{edge general position sets}~\cite{manuel-2022}, \emph{Steiner position sets}~\cite{KlaKuzPetYer}, \emph{mobile position sets}~\cite{KlaKriTuiYer}, \emph{vertex position sets}~\cite{ThaChaTuiThoSteErs}, \emph{lower general position sets}~\cite{lowerposition} and others. Some of the most recent papers on this subject include~\cite{iteration time,KlaTan,KlaTanTian,KorVes,prarenman}.
	
	In this paper we discuss the general position numbers of the Myclieskian of a graph. This construction was introduced by Mycielski in~\cite{mycielski1955coloriage} in order to produce triangle-free graphs with arbitrarily large chromatic number. Properties of the Mycielskian construction have been investigated extensively, for example dominator colouring number~\cite{AbidRao}, connectivity~\cite{Balakrishnan}, energy~\cite{BalKavSo}, $L(2,1)$-labelling number~\cite{DliBouKch}, circular chromatic number~\cite{Fan}, Gromov hyperbolicity~\cite{Granados}, wide diameter~\cite{SavVij} and Italian domination number~\cite{VarLak}, amongst many others.
	
	Let $G$ be a graph with vertex set $V = \{ u_1,\dots ,u_n\} $ and edge set $E$. The Mycielskian of $G$ is the graph $\mathcal{M}(G)$ defined as follows. The vertex set of $\mathcal{M}(G)$ is $V\cup V^\prime \cup \{u^*\}$, where $V = \{ u_1,\dots ,u_n\} $ is the vertex set of $G$, $V^\prime =\{u^\prime : u\in V\}$ is a copy of $V$ and $u^*$ is an additional vertex called the \emph{root vertex}. The edge set of $\mathcal{M}(G)$ is $E \cup \{uv^\prime : uv\in E\} \cup \{v^\prime u^*: v^ \prime \in V^\prime\}$. In other words, the edge set of $\mathcal{M}(G)$ consists of the edges of $G$, and for any $u \in V$ we join $u^\prime $ to the root $u^*$ and to each of the neighbours of $u$ in $G$. The Mycielskian of $P_8$ is shown in Figure~\ref{fig:Mycielskian of path} and the Mycielskian of the cycle $C_5$ in Figure~\ref{fig:Grotzsch graph}.
	
	We note that the Mycielskian is more often denoted by $\mu (G)$, but we avoid this notation here to avoid a clash with the notation used for the mutual visibility number, another position type invariant. We call the vertex $u^\prime$ the \emph{$\mathcal{M}$-twin} of $u$ (and conversely $u$ is the $\mathcal{M}$-twin of $u^\prime $). For any set $X = \{ x_1,\dots ,x_t\}$ of vertices of $G$ we denote the set of $\mathcal{M}$-twins of the vertices in $X$ by $X^\prime = \{ x_1^\prime ,\dots, x_t^\prime\} $. 
	
	\begin{figure}
		\centering
		\begin{tikzpicture}[x=0.4mm,y=-0.4mm,inner sep=0.2mm,scale=0.5,very thick,vertex/.style={circle,draw,minimum size=18,fill=white}]
			\node at (-200,0) [vertex] (u-4) {$u_0$};
			\node at (-150,0) [vertex] (u-3) {$u_1$};
			\node at (-100,0) [vertex] (u-2) {$u_2$};
			\node at (-50,0) [vertex] (u-1) {$u_3$};
			\node at (0,0) [vertex] (u0) {$u_4$};
			\node at (50,0) [vertex] (u1) {$u_5$};
			\node at (100,0) [vertex] (u2) {$u_6$};
			\node at (150,0) [vertex] (u3) {$u_7$};
			\node at (200,0) [vertex] (u4) {$u_8$};
			
			\node at (-200,-100) [vertex] (v-4) {$u_0^\prime $};
			\node at (-150,-100) [vertex] (v-3) {$u_1^\prime $};
			\node at (-100,-100) [vertex] (v-2) {$u_2^\prime $};
			\node at (-50,-100) [vertex] (v-1) {$u_3^\prime $};
			\node at (0,-100) [vertex] (v0) {$u_4^\prime $};
			\node at (50,-100) [vertex] (v1) {$u_5^\prime $};
			\node at (100,-100) [vertex] (v2) {$u_6^\prime $};
			\node at (150,-100) [vertex] (v3) {$u_7^\prime $};
			\node at (200,-100) [vertex] (v4) {$u_8^\prime $};
			
			\node at (0,-200) [vertex] (u*) {$u^*$};
			
			\path
			(u*) edge (v-4)
			(u*) edge (v-3)
			(u*) edge (v-2)
			(u*) edge (v-1)
			(u*) edge (v0)
			(u*) edge (v1)
			(u*) edge (v2)
			(u*) edge (v3)
			(u*) edge (v4)
			
			(u-4) edge (u-3)
			(u-3) edge (u-2)
			(u-2) edge (u-1)
			(u-1) edge (u0)
			(u0) edge (u1)
			(u1) edge (u2)
			(u2) edge (u3)
			(u3) edge (u4)
			
			(v-4) edge (u-3)
			(v-3) edge (u-2)
			(v-2) edge (u-1)
			(v-1) edge (u0)
			(v0) edge (u1)
			(v1) edge (u2)
			(v2) edge (u3)
			(v3) edge (u4)
			
			(v-2) edge (u-3)
			(v-1) edge (u-2)
			(v0) edge (u-1)
			(v1) edge (u0)
			(v2) edge (u1)
			(v3) edge (u2)
			(v4) edge (u3)
			(v-3) edge (u-4)

			;
		\end{tikzpicture}
		\caption{The Mycielskian of a path of length eight}
		\label{fig:Mycielskian of path}
	\end{figure}

	The plan of this paper is as follows. In Section~\ref{sec:complete graphs} we characterise general position sets of $\mathcal{M}(G)$ that contain the root vertex $u^*$. In Section~\ref{sec:MGP partitions} we describe general position sets of $\mathcal{M}(G)$ in terms of a `dual' partition of $V(G)$ into four parts. In Section~\ref{General bound} we give upper and lower bounds for $\gp (\mathcal{M}(G))$ and characterise the graphs that meet our upper bound. We also give exact values of $\gp (\mathcal{M}(G))$ for some common families of graphs. Section~\ref{sec:reg graphs} presents a bound on the gp-number of the Mycielskian of regular graphs, which we use to classify the gp-numbers of Mycielskians of cubic graphs and all sufficiently large regular graphs. Finally in Section~\ref{Trees} we determine the gp-numbers of the Mycielskians of a wide class of trees.
	
	\section{General position sets of $\mathcal{M}(G)$ containing the root $u^*$}\label{sec:complete graphs}
	
	Firstly, we introduce a convention that we use throughout the paper. Given a set $S$ of vertices of a graph, we call a geodesic containing at most two vertices of $S$ \emph{sound} and a shortest path containing at least three vertices of $S$ \emph{unsound}. 
	
	It is useful to disregard gp-sets of $\mathcal{M}(G)$ that contain the root vertex. In this section we classify the graphs such that $\mathcal{M}(G)$ has a general position set of cardinality at least $n+1$ that does not contain $u^*$. We also show that for any graph that is not a complete graph there is always a gp-set of $\mathcal{M}(G)$ that does not contain the root $u^*$.

	\begin{lem}\label{lem:completes}
		Let $G = (V,E)$ be a connected graph with order $n \geq 3$. Then there is a general position set $S$ of $\mathcal{M}(G)$ with $|S| \geq n+1$ that contains $u^*$ if and only if $G$ is a complete graph $K_n$ or the join of $K_1$ with a disjoint union of cliques, at least one of which is $K_1$.
	\end{lem}
	\begin{proof}
		Suppose that $S$ is a gp-set of $\mathcal{M}(G)$ with $|S| \geq n+1$ that contains the root $u^*$. Firstly, since $S\cap V(G)$ cannot contain an induced path $P_3$, the subgraph $\langle S \cap V(G) \rangle $ of $G$ induced by $S\cap V(G)$ must be an independent union of cliques. Also, $S$ can contain at most one vertex of $V^\prime $, since any two vertices of $V^\prime $ are connected by a shortest path through $u^*$.
		
		Suppose that there is a vertex $u \in V(G)$ such that $u,u^\prime \in S$. By our previous observation, $u^\prime $ is the only vertex of $V^\prime $ in $S$. Also, no neighbour of $u$ in $G$ lies in $S \cap V(G)$, since if $u \sim v$, then $u,v,u^\prime $ would be an unsound path. We conclude that $|S| \leq n+2-\deg_G(u)$. Hence $u$ is a leaf of $G$ and $S \cap V(G) = V(G) \setminus \{ v\} $, where $v$ is the support vertex of $u$. Let $W$ be any clique of $\langle S \cap V(G) \rangle $ apart from the leaf $u$. As $G$ is connected, there must an edge from some vertex $w$ of $W$ to a vertex outside $W$; as $u$ is adjacent only to $v$ and $\langle S \cap V(G) \rangle $ is an independent union of cliques, it must be that $w \sim v$. Since $d_{\mathcal{M}(G)}(u,w) = 2$, by the distance-constant property $u$ must be at distance two to every vertex of $W$. Thus $v$ is a universal vertex and $G$ has the claimed structure. It is easily verified that $\{ u^*,u^\prime \} \cup (V(G) \setminus \{ v\} )$ is in general position. An example is shown on the right of Figure~\ref{fig:gpsetu^*}.
		
		Now suppose that there is no vertex $u$ such that $u,u^\prime \in S$. Then we have $|S| = n+1$ and for any vertex $v$ of $G$, $S$ contains exactly one of $v$ and $v^\prime $. If $S \cap V^\prime = \emptyset $, then $V(G) \subset S$ and by our earlier observation $G$ is a clique, in which case $V(G) \cup \{ u^*\} $ is in general position. Otherwise $S$ contains a vertex $w^\prime \in V^\prime $ and so $S = \{ u^*,w^\prime \} \cup (V(G)\setminus \{ w\} )$. However, as $G$ is connected, $w$ has a neighbour $z$ in $G$, and $u^*,w^\prime ,z$ would be an unsound path, a contradiction. 
	\end{proof}
	Notice that the general position sets exhibited in the proof of Lemma~\ref{lem:completes} are not necessarily largest possible. In Corollary~\ref{cor:chargpsetswithu^*} we determine the gp-numbers of the Mycielskian of the join of $K_1$ with a disjoint union of cliques, thereby classifying those graphs for which the Mycielskian has a gp-set containing $u^*$.
	
	\begin{figure}
		\centering
		\begin{tikzpicture}[x=0.2mm,y=-0.2mm,inner sep=0.2mm,scale=0.8,thick,vertex/.style={circle,draw,minimum size=10,fill=lightgray}]
			\node at (-350,-50) [vertex,color=red] (u1) {};
			\node at (-275,-50) [vertex,color=red] (u2) {};
			\node at (-200,-50) [vertex,color=red] (u3) {};
			\node at (-125,-50) [vertex,color=red] (u4) {};
			\node at (-50,-50) [vertex,color=red] (u5) {};
			
			\node at (-350,50) [vertex] (v1) {};
			\node at (-275,50) [vertex] (v2) {};
			\node at (-200,50) [vertex] (v3) {};
			\node at (-125,50) [vertex] (v4) {};
			\node at (-50,50) [vertex] (v5) {};
			
			\node at (-200,120) [vertex,color=red] (u*) {};
			
			\node at (500,-50) [vertex,color=red] (w-1) {};
			\node at (425,-50) [vertex,color=red] (w0) {};
			\node at (350,-50) [vertex,color=red] (w1) {};
			\node at (275,-50) [vertex,color=red] (w2) {};
			\node at (200,-50) [vertex,color=red] (w3) {};
			\node at (125,-50) [vertex] (w4) {};
			\node at (50,-50) [vertex,color=red] (w5) {};
			
			\node at (500,50) [vertex] (z-1) {};
			\node at (425,50) [vertex] (z0) {};
			\node at (350,50) [vertex] (z1) {};
			\node at (275,50) [vertex] (z2) {};
			\node at (200,50) [vertex] (z3) {};
			\node at (125,50) [vertex] (z4) {};
			\node at (50,50) [vertex,color=red] (z5) {};
			
			\node at (200,120) [vertex,color=red] (w*) {};
			
			\path
			
			(w5) edge (w4)
			(w4) edge[bend left] (w3)
			(w4) edge[bend left] (w2)
			(w4) edge[bend left] (w1)
			(w4) edge[bend left] (w0)
			(w4) edge[bend left] (w-1)
			(w3) edge (w2)
			(w1) edge (w0)
			(w1) edge[bend left] (w-1)
			(w0) edge (w-1)
			
			(z4) edge (w-1)
			(z1) edge (w-1)
			(z0) edge (w-1)
			
			(w4) edge (z-1)
			(w1) edge (z-1)
			(w0) edge (z-1)
			
			(w5) edge (z4)
			(w4) edge (z3)
			(w4) edge (z2)
			(w4) edge (z1)
			(w4) edge (z0)
			(w3) edge (z2)
			(w1) edge (z0)
			
			(z5) edge (w4)
			(z4) edge (w3)
			(z4) edge (w2)
			(z4) edge (w1)
			(z4) edge (w0)
			(z3) edge (w2)
			(z1) edge (w0)
			
			(u1) edge (u2)
			(u1) edge[bend left] (u3)
			(u1) edge[bend left] (u4)
			(u1) edge[bend left] (u5)
			(u2) edge (u3)
			(u2) edge[bend left] (u4)
			(u2) edge[bend left] (u5)
			(u3) edge (u4)
			(u3) edge[bend left] (u5)
			(u4) edge (u5)
			
			(u1) edge (v2)
			(u1) edge (v3)
			(u1) edge (v4)
			(u1) edge (v5)
			(u2) edge (v3)
			(u2) edge (v4)
			(u2) edge (v5)
			(u3) edge (v4)
			(u3) edge (v5)
			(u4) edge (v5)
			
			(v1) edge (u2)
			(v1) edge (u3)
			(v1) edge (u4)
			(v1) edge (u5)
			(v2) edge (u3)
			(v2) edge (u4)
			(v2) edge (u5)
			(v3) edge (u4)
			(v3) edge (u5)
			(v4) edge (u5)
			
			(w*) edge (z-1)
			(w*) edge (z0)
			(w*) edge (z1)
			(w*) edge (z2)
			(w*) edge (z3)
			(w*) edge (z4)
			(w*) edge (z5)

			(u*) edge (v1)
			(u*) edge (v2)
			(u*) edge (v3)
			(u*) edge (v4)
			(u*) edge (v5)
			
			;
		\end{tikzpicture}
		\caption{General position sets (in red) used in the proof of Lemma~\ref{lem:completes}}
		\label{fig:gpsetu^*}
	\end{figure} 
	
	\begin{coro}\label{cor:completes}
		For any graph with order $n \geq 3$, the root vertex $u^*$ of $\mathcal{M}(G)$ belongs to every gp-set of $\mathcal{M}(G)$ if and only if $G$ is isomorphic to $K_n$. For any integer $n\geq 3$, $\gp(\mathcal{M}(K_n)) = n+1$ and the unique gp-set of $\mathcal{M}(K_n)$ is $V\cup \{ u^*\} $.
	\end{coro}
	\begin{proof}
		If $\gp (\mathcal{M}(G)) = n$, then $V^\prime $ is a gp-set of $\mathcal{M}(G)$ that does not contain $u^*$, so we can confine our attention to graphs $G$ with $\gp (\mathcal{M}(G)) \geq n+1$. By Lemma~\ref{lem:completes}, if a graph is such that every gp-set of $\mathcal{M}(G)$ contains $u^*$, then $G$ is either a complete graph or the join of $K_1$ with an independent union of cliques, at least one of which consists of a single vertex (moreover, the general position sets exhibited in the proof would have to be gp-sets). If $G$ has the latter structure and $u$ is a leaf with universal support vertex, then $V^\prime \cup \{ u\} $ is a general position set with the same order. Thus, whether or not the general position sets containing $u^*$ are gp-sets, we see that $u^*$ is not contained in every gp-set of such a graph. 
		
		Now suppose that $G$ is a complete graph. It follows from the proof of Lemma~\ref{lem:completes} that the only general position set of $\mathcal{M}(G)$ containing $u^*$ is $V(G) \cup \{ u^*\} $, which has cardinality $n+1$. Suppose that $S$ is a general position set of $\mathcal{M}(G)$ that does not contain $u^*$. Then there must be a vertex $u$ of $K_n$ such that $u,u^\prime \in S$. Now if $v$ is any vertex of $K_n$ other than $u$, the path $u,v,u^\prime $ is a geodesic, so $v \not \in S$. It follows that $S = V^\prime \cup \{ u\} $. However, in this case if $v,w \in V(K_n) \setminus \{ u\} $, the path $v^\prime ,u,w^\prime $ would be unsound, a contradiction. Thus $V(G) \cup \{ u^*\} $ is the unique gp-set of $\mathcal{M}(G)$. 
	\end{proof}
	The gp-set for $\mathcal{M}(K_5)$ is shown on the left of Figure~\ref{fig:gpsetu^*}. Notice that $\gp (\mathcal{M}(K_n)) = n+1$ is true for $n = 2$, as $\mathcal{M}(K_2) \cong C_5$, but the root vertex $u^*$ is not contained in every gp-set.
	
	It follows from Lemma~\ref{lem:completes} that the general position sets containing $u^*$ can easily be found in polynomial time and that for any non-complete graph $G$ there is a gp-set of $\mathcal{M}(G)$ that does not contain. Thus in the remainder of this article we can safely ignore gp-sets containing $u^*$.
	
	\section{General position sets of $\mathcal{M}(G)$ in terms of partitions of $V(G)$}\label{sec:MGP partitions}
	
	In this section we demonstrate a duality between general position sets of $\mathcal{M}(G)$ that do not contain $u^*$ and certain partitions of $V(G)$ into four (possibly empty) parts. This will allow us to calculate $\gp (\mathcal{M}(G))$ working entirely inside $G$. First we observe the connection between shortest paths in the base graph $G$ and the shortest paths in the Mycielskian $\mathcal{M}(G)$. 
	
	\begin{dfn}
		The \emph{projection} of a path $P = u_0,u_1,\dots, u_{\ell }$ in $V(\mathcal{M}(G))\setminus \{ u^*\} $ is the path with all vertices of $P$ that lie in $V^\prime $ replaced by the corresponding $\mathcal{M}$-twin vertices in $V$. For example in Figure~\ref{fig:Mycielskian of path} the projection of the path $u_0,u_1^\prime ,u_2,u_3,u_4^\prime $ is $u_0,u_1,u_2,u_3,u_4$. If $Q$ is a path in $V(G)$, then an \emph{expansion} of $Q$ is any path in $\mathcal{M}(G)$ with projection equal to $Q$.   	
	\end{dfn}
	
	\begin{obs}\label{shortest paths in mu G} The shortest paths in $\mathcal{M}(G)$ between vertices of $V(\mathcal{M}(G))\setminus \{ u^*\} $ have the following form.
		\begin{itemize}
			\item If $u,v \in V(G)$ and $d_G(u,v) \leq 3$, then a $u,v$-path in $\mathcal{M}(G)$ is a shortest path in $\mathcal{M}(G)$ if and only if it is an expansion of a shortest path in $G$.
			\item If $d_G(u,v) \geq 4$ and $P$ is a shortest $u,v$-path in $\mathcal{M}(G)$, then either $d_G(u,v) = 4$ and $P$ is an expansion of a $u,v$-path of length four in $G$, or else $P$ has the form $u,w^\prime _1,u^*,w^\prime _2,v$, where $u \sim w_1$ and $v \sim w_2$ in $G$. 
			\item If $u \not = v$, then the shortest $u,v^\prime $-paths in $\mathcal{M}(G)$ are either expansions of shortest $u,v$-paths in $G$ or have the form $u,w^\prime ,u^*,v^\prime $, where $u \sim w$ in $G$. 
			\item All shortest $u,u^\prime $-paths in $\mathcal{M}(G)$ have the form $u,w,u^\prime $, where $w \sim u$ in $G$.
		\end{itemize}
	\end{obs}	
	
	With any general position set $S$ of $\mathcal{M}(G)$ that does not contain $u^*$ we associate the partition $\pi (S) = (V_1,V_2,V_3,V_4)$ of $V(G)$ where:
	\begin{itemize}
		\item $V_1 = \{ u \in V(G): u,u^\prime  \in S\}$,
		\item $V_2 = \{ u \in V(G): u \not \in S, u^\prime  \in S\} $,
		\item $V_3 = \{ u \in V(G): u \in S, u^\prime  \not \in S \} $, and
		\item $V_4 = \{ u \in V(G): u,u^\prime  \not \in S\} $.
	\end{itemize}
	Then we have $\gp (\mathcal{M}(G)) = n+n_1-n_4$, where $n$ is the order of the graph and $n_i = |V_i|$ for $i = 1,2,3,4$. We now identify the essential properties of such a partition that guarantee it has the form $\pi (S)$, where $S$ is a general position set of $\mathcal{M}(G)$. We call such a partition an \emph{$\mathcal{MGP}$-partition} (short for `Mycielskian general position').
	\begin{dfn}\label{dfn:Mycielski partition}
		A partition $\pi = (V_1,V_2,V_3,V_4)$ of $V(G)$ into four (possibly empty) sets is an \emph{$\mathcal{MGP}$-partition} if and only if it satisfies the following three conditions:
		\begin{enumerate}
			\item if $e$ is an edge in $\langle V_1 \cup V_3\rangle $, then $e$ has both endpoints in $V_3$, 
			\item if a vertex $u$ in $V_1 \cup V_3$ has a neighbour $v$ in $V_2$, then $d_G(u,w) = 2$ for all $w \in (V_1 \cup V_2) \setminus \{ u,v\} $ and $d_G(u,w) \leq 3$ for all $w \in V_3\setminus\{ u\} $, and 
			\item if $P$ is a shortest path $u_0,u_1,\dots, u_{\ell }$ with length $\ell \leq 4$ in $G$ that passes through three or more vertices of $V_1 \cup V_2 \cup V_3$, then either $u_0,u_{\ell } \in V_2$ and $\ell \geq 3$, or else $P$ or $\widetilde{P}$ has one of the following forms:
			\begin{enumerate}
				\item $u_0,u_1,u_2$, where $u_0,u_1 \in V_2$, 
				\item $u_0,u_1,u_2,u_3$, where $u_0,u_1 \in V_2$, $u_2 \notin V_1 \cup V_2 \cup V_3$ and $u_3 \in V_1 \cup V_3$, or
				\item $u_0,u_1,u_2,u_3,u_4$, where $u_0 \in V_2$, $u_3 \not \in V_1 \cup V_2$ and $u_4 \in V_1 \cup V_3$.
			\end{enumerate}
		\end{enumerate}
	\end{dfn}
	
	Note the following three important properties of $\mathcal{MGP}$-partitions.
	
	\begin{lem}\label{cor:partition}
		If $(V_1,V_2,V_3,V_4)$ is an $\mathcal{MGP}$-partition of $G$, then $V_1$ is an independent set of $G$, $(V_1,V_3) = \emptyset $ and $(V_1,V_2)$ is a matching.
	\end{lem}
	\begin{proof}
		That $V_1$ is independent and $(V_1,V_3) = \emptyset $ follows immediately from Condition 1 in Definition~\ref{dfn:Mycielski partition}. We now show that $(V_1,V_2)$ is a matching. By Condition 2 of Definition~\ref{dfn:Mycielski partition}, a vertex $u \in V_1$ can be at distance one from at most one vertex of $V_2$, so each vertex of $V_1$ has at most one neighbour in $V_2$. Suppose now that there is a vertex $w \in V_2$ with distinct neighbours $u,v \in V_1$. As $V_1$ is independent, $u,w,v$ would be a shortest path in $G$ of the form forbidden by Condition 3. It follows that the edges in $(V_1,V_2)$ are independent.  
	\end{proof}
	
	Given an $\mathcal{MGP}$-partition $\pi = (V_1,V_2,V_3,V_4)$, we associate the subset $\sigma (\pi )$ given by $V_1 \cup V_1^\prime \cup V_2^\prime \cup V_3$.
	
	\begin{lem}
		If $S$ is a general position set of $\mathcal{M}(G)$ that does not contain $u^*$, then $\pi (S)$ is an $\mathcal{MGP}$-partition.
	\end{lem}
	\begin{proof}
		We verify the three properties from Definition~\ref{dfn:Mycielski partition} in turn.
		\newline
		\textbf{Condition 1: }  If $u \in V_1$, $v \in V_1 \cup V_3$ and $u \sim v$ in $G$, then the shortest path $u,v,u^\prime $ in $\mathcal{M}(G)$ would be unsound. Thus any edge of $G$ in $\langle V_1 \cup V_3\rangle $ has both endpoints in $V_3$.
		
		\textbf{Condition 2: } Suppose that a vertex $u \in V_1 \cup V_3$ has a neighbour $v \in V_2$. By Condition 1, $u$ has no neighbours in $V_1$. Furthermore, if $u$ had a second neighbour $v_2 \not = v$ in $V_2$, then the path $v^\prime ,u,v_2^\prime $ would be unsound. Thus $d_G(u,w) \geq 2$ for each $w \in (V_1 \cup V_2)\setminus \{u,v\} $. If there is a $w \in (V_1 \cup V_2) \setminus \{ u,v\} $ such that $d_G(u,w) \geq 3$, then $u,v^\prime ,u^*,w^\prime $ would be an unsound geodesic. Therefore $d_G(u,w) = 2$ for all $w \in (V_1 \cup V_2)\setminus \{ u,v\} $. Also if $w \in V_3 \setminus \{ u\} $ is such that $d_G(u,w) \geq 4$, then $u,v^\prime ,u^*,w_1^\prime ,w$ would be an unsound path, where $w_1$ is any neighbour of $w$ in $G$. 
		
		\textbf{Condition 3: } Suppose that $P$ is a geodesic $u_0,u_1,\dots ,u_{\ell }$ in $G$ with length $\ell \leq 4$ that has both endpoints and an internal vertex in $V_1 \cup V_2 \cup V_3$. If $P$ has both endpoints in $V_1 \cup V_3$, then $P$ has an unsound expansion that is a geodesic in $\mathcal{M}(G)$, so we can assume that the initial vertex $u_0$ of $P$ is in $V_2$. Suppose firstly that $\ell = 2$. Then we must have $u_1 \in V_2$, for if $u_1 \in V_1 \cup V_3$, then $u_0^\prime ,u_1,u_2$ is unsound if $u_2 \in V_1 \cup V_3$ and $u_0^\prime ,u_1,u_2^\prime $ is unsound if $u_2 \in V_2$. 
		
		Now suppose that $\ell = 3$ and $u_3 \in V_1 \cup V_3$. As $u_0^\prime ,u_1,u_2,u_3$ and $u_0^\prime ,u^*,u_2^\prime ,u_3$ are shortest paths in $\mathcal{M}(G)$, we see that $u_1,u_2 \not \in V_1 \cup V_3$ and $u_2 \not \in V_2$, so that $P$ must have the form in variant b) of Condition 3. Lastly, if $\ell = 4$ and $u_4 \in V_1 \cup V_3$, then $u_0^\prime ,u^*,u_3^\prime ,u_4$ is a shortest path in $\mathcal{M}(G)$, so that $u_3 \not \in V_1 \cup V_2$ and $P$ has the form of variant c). 
	\end{proof}

	We now complete our characterisation of gp-sets of Mycielskians of graphs. 
	
	\begin{thm}\label{thm:characterisation}
		Let $\pi = (V_1,V_2,V_3,V_4)$ be an $\mathcal{MGP}$-partition of $G$. Then $\sigma (\pi ) = V_1 \cup V_1^\prime \cup V_2^\prime \cup V_3$ is a general position set of $\mathcal{M}(G)$. 
	\end{thm}
	\begin{proof}
		We need to verify that all shortest paths between vertices of $S = \sigma (\pi )$ are sound. First let $u,v \in V_1 \cup V_3$. By Condition 3 of Definition~\ref{dfn:Mycielski partition}, no expansion of a shortest $u,v$-path in $G$ with length at most four contains a third vertex of $S$. Furthermore, if $d_G(u,v) \geq 4$, then by Observation~\ref{shortest paths in mu G} the only other geodesics to consider are the paths of the form $u,w^\prime _1,u^*,w^\prime _2,v$, where $u \sim w_1$ and $v \sim w_2$ in $G$. By Condition 1 of Definition~\ref{dfn:Mycielski partition} neither of the vertices $w_1$ or $w_2$ is in $V_1$. Hence we can assume that $w_1 \in V_2$. Then by Condition 2 we would have $d_G(u,v) \leq 3$, so that $u,w^\prime _1,u^*,w^\prime _2,v$ would not be a shortest path. 
		
		Now suppose that $u^\prime ,v^\prime  \in V_1^\prime  \cup V_2^\prime $. A shortest $u^\prime ,v^\prime $-path $P$ in $\mathcal{M}(G)$ has length two and $P$ is unsound only if it passes through a common neighbour of $u$ and $v$ in $V_1 \cup V_3$. By Condition 2 and Lemma~\ref{cor:partition} this is impossible. 
		
		We can thus assume that $P$ is an unsound geodesic with one endpoint in $V_1 \cup V_3$ and the other in $V_1^\prime \cup V_2^\prime $. If $u \in V_1$, the $u,u^\prime $-geodesics in $\mathcal{M}(G)$ are the paths $u,w,u^\prime $, where $u \sim w$ in $G$. By Lemma~\ref{cor:partition} the vertex $w$ does not lie in $S$. Hence the endpoints of $P$ are not $\mathcal{M}$-twins of each other. Let us denote the endpoints by $u,v^\prime $, where $u \in V_1 \cup V_3$, $v \in V_1 \cup V_2$ and $u \not = v$. If $\ell (P) = 2$, then by Condition 1 of Definition~\ref{dfn:Mycielski partition} we must have $u \in V_3$, $v \in V_2$ and the internal vertex of $P$ lies in $V_3$. However, this form of shortest path is not allowed by Condition 3.
		
		Hence $\ell (P) = 3$. By Observation~\ref{shortest paths in mu G}, the $u,v^\prime $-geodesics in $\mathcal{M}(G)$ are either expansions of shortest paths of length three in $G$, or else have the form $u,w^\prime ,u^*,v^\prime $, where $u \sim w$ in $G$. By Condition 3 the former geodesics are sound and the latter are sound by Condition 2. This shows that all shortest paths in $\mathcal{M}(G)$ are sound and thus $S$ is in general position.   
	\end{proof}
	Since both maps $\pi $ and $\sigma $ are one-to-one, this completes the proof of the claimed duality between general position sets of $\mathcal{M}(G)$ not containing $u^*$ and $\mathcal{MGP}$-partitions of $G$.
	
	\begin{coro}\label{cor:characterisation}
		Let $G$ be a non-complete graph with order $n$. Then $\gp (\mathcal{M}(G))$ is the largest value of $n+n_1-n_4$ over all $\mathcal{MGP}$-partitions $(V_1,V_2,V_3,V_4)$ of $G$, where $n_i = |V_i|$ for $1 \leq i \leq 4$. 
	\end{coro}

	\section{Bounds, extremal cases and exact values}\label{General bound}
	
	In this section we use Corollary~\ref{cor:characterisation} to derive tight upper and lower bounds for $\gp (\mathcal{M}(G))$ and characterise the case of equality with the upper bound. We start with the lower bound. Recall that $\ip (G)$ is the number of vertices in a largest independent set that is also in general position and $\gp _d(G)$ is the number of vertices in a largest subset $S \subseteq V(G)$ such that no shortest path with length at most $d$ passes through three or more vertices in $S$. To derive bounds for $\gp (\mathcal{M}(G))$ we combine these concepts in the following definition.
	
	\begin{dfn}
		An \emph{independent $d$-position set} is a subset $S \subseteq V(G)$ such that $S$ is an independent set and no shortest path of length at most $d$ passes through three or more vertices of $S$. We denote that largest number of vertices in an independent $d$-position set by $\ip _d(G)$.
	\end{dfn}
	For any graph $G$ these parameters satisfy the inequalities $\ip (G) \leq \gp (G)$ and $\ip (G) \leq \ip _4(G) \leq \alpha (G)$.
	\begin{coro}\label{thm:lower bound}
		For any graph $G$ with order $n \geq 3$,
		\[ \gp (\mathcal{M}(G)) \geq \max \{ n,2\ip _4(G)\} .\] 
	\end{coro}
	\begin{proof}
		The set $V^\prime $ is trivially in general position, implying that $\gp (\mathcal{M}(G)) \geq n$ (this is equivalent to using the $\mathcal{MGP}$-partition $V_2 = V(G)$). Let $A$ be a largest independent $4$-position set of $G$. Setting $V_1 = A$ and $V_4 = V(G)\setminus A$ gives an $\mathcal{MGP}$-partition of $G$, yielding $\gp (\mathcal{M}(G)) \geq 2\ip _4(G)$ by Corollary~\ref{cor:characterisation}. 
	\end{proof}
	We call any non-complete graph $G$ with order $n \geq 3$ and $\gp (\mathcal{M}(G)) = \max \{ n,2\ip _4(G)\}$ \emph{meagre}. Otherwise $G$ is \emph{abundant}. We now show that the lower bound in Corollary~\ref{thm:lower bound} is tight by exhibiting two families of meagre graphs: graphs with large size, joins of $K_1$ with disjoint unions of cliques of cardinality at least three, and complete multipartite graphs. 
	
	\begin{lem}\label{lem:abundant,V_4 empty}
		If $G$ is an abundant graph with order $n \geq 3$ and $V_4 = \emptyset $, then $V_3 = \emptyset $, $|V_1| = 1$ and the vertex of $V_1$ is a leaf with support vertex that is a universal vertex. If $G$ has a leaf with a universal support vertex, then $\gp (\mathcal{M}(G)) \geq n+1$.
	\end{lem}
	\begin{proof}
		Suppose that $G$ is abundant and has $\mathcal{MGP}$-partition $(V_1,V_2,V_3,\emptyset )$, where $n_1 \geq 1$. Suppose that $V_3 \not = \emptyset $. As we are dealing with connected graphs, each vertex of $V_1$ has a path to $V_3$. Therefore, by Lemma~\ref{cor:partition}, the matching $(V_1,V_2)$ saturates $V_1$. Then by Condition 2 of Definition~\ref{dfn:Mycielski partition} we have $d(u,w) \leq 3$ for each $u \in V_1$ and $w \in V_3$. Therefore there is a shortest path from some $u \in V_1$ to a $w \in V_3$ with length two or three with all internal vertices in $V_2$. However, the existence of such a path contradicts Condition 3 of Definition~\ref{dfn:Mycielski partition}. Thus $V_3 = \emptyset $ and $V(G) = V_1 \cup V_2$. Thus each vertex of $V_1$ is a leaf.
		
		By Condition 2 of Definition~\ref{dfn:Mycielski partition} any vertex of $V_1$ is at distance two from every vertex of $V_2$, with the exception of its support vertex. Therefore any pair of vertices from $V_1$ are at distance three from each other, with shortest path passing through $V_2$. This contradicts Condition 3 of Definition~\ref{dfn:Mycielski partition} unless $n_1 = 1$. By Condition 2 of Definition~\ref{dfn:Mycielski partition} it follows that the support vertex of the vertex of $V_1$ is universal. 
		
		Conversely, if $u \in V(G)$ is a leaf with a universal support vertex, then $(\{ u\} ,V(G)\setminus \{ u\} ,\emptyset ,\emptyset )$ is an $\mathcal{MGP}$-partition and $n+n_1-n_4 = n+1$.   
	\end{proof}

	\begin{thm}
		For $n \geq 3$, the largest size of a graph with order $n$ that satisfies $\gp(\mathcal{M}(G)) > n$ is ${n \choose 2}+1$. Any abundant graphs with size ${n \choose 2}+1$ is either a clique $K_{n-1}$ with a leaf attached, the gem graph $K_1 \vee P_4$, or $3K_1 \vee K_2$. The latter graph is not abundant, for the former two are. 
	\end{thm}
	\begin{proof}
		Let $G$ be an abundant graph with order $n$ and size at least ${n \choose 2}+1$. Let $(V_1,V_2,V_3,V_4)$ be an $\mathcal{MGP}$-partition corresponding to a gp-set of $\mathcal{M}(G)$. By Lemma~\ref{cor:partition} each vertex $u$ of $V_1$ has degree at most $1+n_4$ if $u$ has an edge to $V_2$ and degree at most $n_4$ otherwise. By Lemma~\ref{lem:abundant,V_4 empty}, if $n_1 \geq 2$, then $n_1 > n_4 > 0$.
		
		Suppose that $n_1 \geq 3$ and let $u_1,u_2,u_3 \in V_1$. Let $t= |(\{ u_1,u_2,u_3\} ,V_2)|$.  Then, taking care not to count missing edges between $u_1$, $u_2$ and $u_3$ twice, we see that there are at least \[ 3(n-n_4-1)-t-3 = (n-2)+(2n-3n_4-t-4)\] missing edges in $G$. As $n_2 \geq t$ and $n_1 > n_4$, we have
		\[ 2n-3n_4-t-4 = 2n_1+2n_2+2n_3-n_4-t-4 \geq n_1+n_2+2n_3-3 \geq 0.\] Moreover, for equality to hold, we would need $n_1 = 3, n_2 = n_3 = 0, n_4 = 2$. As all other edges must be present, $G$ is isomorphic to $3K_1 \vee K_2$ and $V_1$ is an independent $4$-position set, so that we would have $\gp (\mathcal{M}(G)) = 2\ip _4(G)$ and $G$ would not be abundant.
		
		Now suppose that $n_1 = 2$, where $V_1 = \{ u_1,u_2\} $. Hence $n_4 = 1$ and $n = n_2+n_3+3$. Set $t = |(V_1,V_2)|$. There are at least $2n_2+2n_3+1-t$ missing edges with endpoints in $V_1$. We have
		\[ 2n_2+2n_3+1-t = (n-2)+(n_2+n_3-t) \geq n-2,\] since $n_2 \geq t$. To have equality, we must have $n_2 = t$ and $n_3 = 0$. As all other edges must be present, if $t = 0,1$, then $G$ is a clique with a leaf attached. If $t = 2$, then $G$ is isomorphic to the gem graph, which is readily verified to be abundant, since its independence number is two. A gp-set of the Mycielskian of the gem graph is shown in Figure~\ref{fig:gem}.
		
		Finally if $n_1 = 1$, then we must have $n_4 = 0$ and Lemma~\ref{lem:abundant,V_4 empty} tells us that the vertex of $V_1$ is a leaf. The edges missing to the leaf account for all $n-2$ missing edges, so $G$ is a clique with a leaf attached as claimed.
	\end{proof}

	\begin{figure}
		\centering
		\begin{tikzpicture}[x=0.2mm,y=-0.2mm,inner sep=0.2mm,scale=0.8,thick,vertex/.style={circle,draw,minimum size=10,fill=lightgray}]
			\node at (-200,-50) [vertex,color=red] (u1) {};
			\node at (-100,-50) [vertex,color=red] (u2) {};
			\node at (0,-50) [vertex] (u3) {};
			\node at (100,-50) [vertex] (u4) {};
			\node at (200,-50) [vertex] (u5) {};
			
			\node at (-200,50) [vertex,color=red] (v1) {};
			\node at (-100,50) [vertex,color=red] (v2) {};
			\node at (0,50) [vertex,color=red] (v3) {};
			\node at (100,50) [vertex,color=red] (v4) {};
			\node at (200,50) [vertex] (v5) {};
			
			\node at (0,120) [vertex] (u*) {};
			
			\path
			
			(u1) edge[bend left] (u3)
			(u2) edge[bend left] (u4)
			(u1) edge[bend left] (u5)
			(u2) edge[bend left] (u5)
			(u3) edge (u4)
			(u4) edge (u5)
			(u3) edge[bend left] (u5)
			
			(v1) edge (u3)
			(v2) edge (u4)
			(v1) edge (u5)
			(v2) edge (u5)
			(v3) edge (u4)
			(v4) edge (u5)
			(v3) edge (u5)
			
			(u1) edge (v3)
			(u2) edge (v4)
			(u1) edge (v5)
			(u2) edge (v5)
			(u3) edge (v4)
			(u4) edge (v5)
			(u3) edge (v5)
			
			(u*) edge (v1)
			(u*) edge (v2)
			(u*) edge (v3)
			(u*) edge (v4)
			(u*) edge (v5)
			
			;
		\end{tikzpicture}
		\caption{A gp-set (in red) of the Mycielskian of the gem graph}
		\label{fig:gem}
	\end{figure}

	\begin{thm}
		Let $K_{r_1,r_2,\dots,r_k}$ be a complete $k$-partite graph, where $2 \leq k < n = r_1+\dots +r_k$ and $r_1 \geq r_2 \geq \dots \geq r_k$. Then\[ \gp (\mathcal{M}(K_{r_1,r_2,\dots,r_k}))  = \max \{ n,2r_1\} .\]
	\end{thm}
	\begin{proof}
		Since any partite set of $G = K_{r_1,r_2,\dots,r_k}$ is an independent position set, by Corollary~\ref{thm:lower bound} we have $\gp (\mathcal{M}(G))  = \max \{ n,2r_1\} $. Suppose for a contradiction that $(V_1,V_2,V_3,V_4)$ is an $\mathcal{MGP}$-partition of $G$ with $n_i = |V_i|$ for $i = 1,2,3,4$ corresponding to a gp-set of $\mathcal{M}(G)$ containing more than $\max \{ n,2r_1\}$ vertices.
		
		As $V_1$ is an independent set, it must be contained entirely within some partite set $X$ of $G$. By Lemma~\ref{cor:partition} there are no edges from $V_1$ to $V_3$, so we have $V_1 \cup V_3 \subseteq X$. If $V_2$ lies entirely in $X$, then $\gp (\mathcal{M}(K_{r_1,r_2,\dots,r_k})) = 2n_1+n_2+n_3\leq 2r_1$, so assume that there is a vertex $y \in V_2\setminus X$. Then the vertex $y$ is adjacent to every vertex of $X$. As $(V_1,V_2)$ is a matching by Lemma~\ref{cor:partition}, we must have $n_1 = 1$ and hence $n_4 = 0$. It now follows from Lemma~\ref{lem:abundant,V_4 empty} that the vertex of $V_1$ is a leaf, so there is just one vertex of $G$ outside $X$ and hence $G$ is a star, i.e. a tree of order $n$ with a universal vertex. The associated general position set of $\mathcal{M}(G)$ contains $n+1$ vertices. However, for $n \geq 3$ this would be smaller than $2r_1 = 2(n-1)$, a contradiction. It follows that the graph is meagre.
	\end{proof}

	\begin{thm}\label{thm:joinofK1withunionofcliques}
		Let $G$ be the join of $K_1$ with with a disjoint union of $t \geq 2$ cliques $W = \dot {\bigcup } _{i = 1}^t W_i$. Suppose that $W$ contains $t_1$ cliques of order at most two and $t_2$ cliques of order at least three. Then $\gp (\mathcal{M}(G)) = n+t_1$.
	\end{thm}
	\begin{proof}
		Let $(V_1,V_2,V_3,V_4)$ be an $\mathcal{MGP}$-partition corresponding to a gp-set $S$ of $\mathcal{M}(G)$ not containing the root $u^*$. It follows from Lemma~\ref{lem:completes} that $|S| \geq n+1$, so we can assume that $n_1 > n_4$. Let $u$ be the universal vertex of $G$. If $u \in V_3$, then by Definition~\ref{dfn:Mycielski partition} $V_1$ is empty, contradicting $|S| > n$. If $u \in V_1$, then by Condition 1 of Definition~\ref{dfn:Mycielski partition} there are no further vertices in $V_1 \cup V_3$. By Condition 3, there cannot be vertices of $V_2$ in two different cliques of $W$, so $V_4$ is non-empty and we would have $|S| \leq n$, a contradiction. If $u \in V_2$, then, assuming that $W_1$ contains a vertex $v$ of $V_1$, Definition~\ref{dfn:Mycielski partition} shows that $G$ contains no further vertices of $V_1$, so $V_4 = \emptyset $ and Lemma~\ref{lem:abundant,V_4 empty} shows that $|W_1| = 1$ and $|S| = n+1$. Hence we can assume that $u \in V_4$.
		
		Any clique $W_i$ contains at most one vertex of $V_1$, so can contribute at most one to the difference $n_1-n_4$. By Definition~\ref{dfn:Mycielski partition}, if any of the $t_2$ cliques with order at least three contains a vertex of $V_1$, then it must also contain a vertex of $V_4$, and so cannot contribute to $n_1-n_4$. Hence $|S| \leq n+t_1$. Conversely, setting $V_4 = \{ u\} $, choosing one vertex from each of the $t_1$ cliques of order at most two to be in $V_1$ and setting all other vertices of $W$ to belong to $V_2$ gives an $\mathcal{MGP}$-partition. Comparing the cardinality of the general position sets obtained, including those containing the root $u^*$ from Lemma~\ref{lem:completes}, we conclude that $\gp (\mathcal{M}(G)) = n+t_1$.
	\end{proof}
	Therefore we obtain a family of meagre graphs by letting all of the cliques in $W$ have order at least three.
	\begin{coro}\label{cor:chargpsetswithu^*}
		A graph $G$ has a gp-set containing the root vertex $u^*$ if and only if $G$ is a clique or $G$ is the join of $K_1$ with a disjoint union of $t \geq 2$ cliques $W_1,\dots ,W_t$, where one of the cliques has order one and all of the other cliques have order at least three.
	\end{coro}
	
	Later in the paper we will see other examples of meagre graphs: paths (Lemma~\ref{lem:meagre tree}), all cycles apart from $C_3$ and $C_5$ (Theorem~\ref{thm: cycles}), all cubic graph with two exceptions (Theorem~\ref{thm:cubic graphs}) and in general all sufficiently large regular graphs (Theorem~\ref{thm:largeregulargraphsaremeagre}). It also follows from Corollary~\ref{bound} below that for $n \geq 2$ the star graph $S_n$ is meagre (in fact for this graph the bounds in Corollaries~\ref{thm:lower bound} and~\ref{bound} coincide). We now derive a tight upper bound on $\gp (\mathcal{M}(G))$ in terms of the independence number $\alpha (G)$.
	
	\begin{thm}\label{thm:upperbound}
		For any non-complete graph $G$ with order $n$ and minimum degree $\delta \geq 1$ we have
		\begin{equation}\label{eqn:upper bound}
			\gp(\mathcal{M}(G)) \leq n+\max\{ 0,\ip _4(G)-\delta +1\} .
		\end{equation} 
	\end{thm}
	\begin{proof}
		Let $(V_1,V_2,V_3,V_4)$ be the $\mathcal{MGP}$-partition associated with a gp-set $S$ of $\mathcal{M}(G)$. If $V_1$ is empty, then $\gp(\mathcal{M}(G)) \leq n$, so assume that $V_1 \not = \emptyset $. It follows from the conditions in Definition~\ref{dfn:Mycielski partition} and Lemma~\ref{cor:partition} that $V_1$ is an independent 4-position set of $G$. Thus $n_1 \leq \ip _4(G)$. Also by Lemma~\ref{cor:partition}, any vertex in $V_1$ has at most one neighbour in $V_2$ and hence has at least $\delta - 1$ neighbours in $V_4$. Thus $n_4 \geq \delta - 1$ and so $\gp(\mathcal{M}(G)) = n+n_1-n_4 \leq n+\ip _4(G)-\delta +1$. To meet this bound $V_1$ must be a largest independent 4-position set of $G$, which implies that $n_1 \geq 2$, and each vertex of $V_1$ must have a neighbour in $V_2$, so by Condition 2 of Definition~\ref{dfn:Mycielski partition} any pair of vertices in $V_1$ has a common neighbour in $V_4$, implying that $\delta \geq 2$. Thus if $\delta = 1$ we can improve the bound to $\gp(\mathcal{M}(G)) \leq n+\ip _4(G)-1$.
	\end{proof}
	Theorem~\ref{thm:upperbound} also gives an upper bound in terms of the more familiar independence number.
	
	\begin{coro}\label{bound}
		For any non-complete graph $G$, $\gp (\mathcal{M}(G)) \leq n + \alpha -1$.
	\end{coro}
	
	We now characterise the graphs that meet these upper bounds.
	
	\begin{thm}\label{thm:n+ip_4-delta+1characterisation}
		An abundant graph $G$ satisfies $\gp (\mathcal{M}(G)) = n+\ip _4(G)-\delta +1$ if and only if $V(G)$ can be partitioned into three sets $V_1$, $V_2$ and $V_4$ such that 
		\begin{itemize}
			\item $|V_1| = |V_2| > |V_4| = \delta - 1$,
			\item $V_1$ is an independent set,
			\item $(V_1,V_2)$ is a matching,
			\item the vertices of $V_2$ have degree at least $\delta $, and
			\item any vertex of $V_2$ that is not universal in $\langle V_2 \rangle $ has an edge to $V_4$.
		\end{itemize}
	\end{thm}
	\begin{proof}
		Suppose that $G$ meets the upper bound in Theorem~\ref{thm:upperbound}. Let $(V_1,V_2,V_3,V_4)$ be an $\mathcal{MGP}$-partition associated with a gp-set of $G$. The proof of Theorem~\ref{thm:upperbound} shows that $V_1$ is a largest independent 4-position set, the matching $(V_1,V_2)$ saturates $V_1$, each vertex of $V_1$ has degree $\delta $ and $n_4 = \delta -1$. Write $V_1 = \{ u_1,\dots ,u_r\} $ and $V_2 = \{ v_1,\dots ,v_s\} $, where $s \geq r$ and $u_i \sim v_i$ for $1 \leq i \leq s$. Hence for $1 \leq i \leq r$ we have $N_G(u_i) = \{ v_i\} \cup V_4$. 
		
		We have $\ip _4(G) \geq 2$ for any non-complete graph with order $n \geq 3$, so $n_1 \geq 2$. By Condition 2 of Definition~\ref{dfn:Mycielski partition}, any pair of vertices in $V_1$ has a common neighbour in $V_4$, implying that $\delta \geq 2$. Thus $V_4$ is non-empty.
		
		Consider the vertices of $V_3$. As $V_1$ is a maximum independent 4-position set and $(V_1,V_3) = \emptyset $, it follows for that any vertex $w$ of $V_3$ there is a shortest path $P$ of length four containing $w$ and two vertices of $V_1$. Observe that any pair of vertices from $V_1$ are at distance two from each other, since they have a common neighbour in $V_4$. Also by Condition 2 of Definition~\ref{dfn:Mycielski partition} each vertex of $V_1$ is at distance at most three from any vertex of $V_3$. It follows that there is no such path $P$ and $V_3 = \emptyset $. Similarly, if $s > r$ and $v$ is a vertex of $\{ v_{r+1},\dots ,v_s\} $, then there would have to be a shortest path in $G$ of length four that contains $v$ and two vertices of $V_1$. This is impossible, since the vertices of $V_1$ are at distance two apart and by Condition 2 of Definition~\ref{dfn:Mycielski partition} each vertex of $V_1$ is at distance two from $v$. Therefore $r = s$.  
	\end{proof}

	\begin{coro}\label{n+alpha-1 characterisation}
		A non-complete graph $G$ with order $n \geq 3$ has $\gp(\mathcal{M}(G)) = n+\alpha(G)-1$ if and only if $G$ is either a clique $K_{n-1}$ with a leaf attached, or has the following structure:
		\begin{itemize}
			\item $V(G) = V_{1,1} \cup V_{1,2} \cup V_2 \cup \{ x\}$,
			\item $V_{1,1} \cup V_{1,2}$ is an independent set,
			\item $x$ is adjacent to every vertex of $V_{1,1} \cup V_{1,2}$,
			if $V_{1,1} \not = \emptyset $, then $x$ is a universal vertex,
			\item $(V_{1,2},V_2)$ is a perfect matching, and
			\item any vertex of $V_2$ that is not universal in $\langle V_2 \rangle $ is adjacent to $x$.
		\end{itemize}
	\end{coro}
	\begin{proof}
		It follows from the proof of Theorem~\ref{thm:n+ip_4-delta+1characterisation} that if any vertex of $V_1$ has degree at least three, then
		\[ \gp (\mathcal{M}(G)) \leq n+\ip _4(G) -3+1 \leq n +\alpha (G)-2.\]
		Hence we can assume that each vertex of $V_1$ has degree one or two. To have $\gp (\mathcal{M}(G)) = n+\alpha -1$, either a) $V_1$ is a maximum independent set and $n_4 = 1$, or b) $V_1$ is an independent set of order $\alpha - 1$ and $V_4 = \emptyset $. 
		
		In case b) we have $\alpha = 2$ and Lemma~\ref{lem:abundant,V_4 empty} tells us that there are vertices $u,v \in V(G)$ such that $u$ is a leaf and $v$ is a universal vertex, and $V_1 = \{ u\} $, $V_2 = V(G) \setminus \{ u\} $ and $V_3 = \emptyset $. Also, if there is any pair of vertices $w,z \in V_2 \setminus \{ v\} $ that are non-adjacent, then the independence number of $G$ would be at least three. Thus $G$ is a clique $K_{n-1}$ with a leaf attached, which meets the upper bound $n+\alpha -1$ by Theorem~\ref{thm:joinofK1withunionofcliques}.
		
		Now consider option b). As $V_1$ is a maximum independent set, every vertex of $V_2 \cup V_3 \cup V_4$ has an edge to $V_1$. By Definition~\ref{dfn:Mycielski partition} it follows that $V_3 = \emptyset $ and the matching $(V_1,V_2)$ saturates $V_2$. Let $x$ be the vertex of $V_4$. If $V_2 = \emptyset $, then each vertex of $V_1$ is a leaf attached to $x$ and $G$ is a star. The $n-1$ leaves of the star form an independent 4-position set, so by Corollary~\ref{thm:lower bound} the star meets the upper bound $n+\alpha -1$. Hence assume that $V_2 \not = \emptyset $.
		
		We split $V_1$ into three parts $V_{1,1}$, $V_{1,2}$ and $V_{1,3}$, where 
		\begin{itemize}
			\item $V_{1,1}$ are the vertices of $V_1$ adjacent only to $x$,
			\item $V_{1,2}$ are the vertices of $V_1$ with degree two, and
			\item $V_{1,3}$ are the vertices of $V_1$ adjacent only to a vertex $V_2$.
		\end{itemize}
		As the matching $(V_1,V_2)$ saturates $V_2$, there is a perfect matching between $V_{1,2} \cup V_{1,3}$ and $V_2$. If $u \in V_{1,3}$, then there is no path of length two from $u$ to any vertex of $V_1 \setminus \{ u\} $, so that we would have $V_1 = \{ u\} $ by Condition 2 of Definition~\ref{dfn:Mycielski partition}. As $(V_{1,2} \cup V_{1,3},V_2)$ is a perfect matching, $V_2$ would also have just one vertex and $G$ must be the path $P_3$. Upon setting $V_2 = \emptyset $ and $|V_{1,1}| = 2$ in the statement of the corollary, $P_3$ has the claimed form. Hence we can now assume that $V_{1,3} = \emptyset $. 
		
		Suppose that there is a pair of vertices $v_1,v_2$ of $V_2$ such that $v_1 \not \sim v_2$, and let $u_1,u_2$ be the neighbours of $v_1,v_2$ in $V_1$ respectively. Then by Condition 2 of Definition~\ref{dfn:Mycielski partition} we have $d_G(u_1,v_2) = d_G(u_2,v_1) = 2$, so that both $v_1$ and $v_2$ are adjacent to $x$. It follows that any vertex $v \in V_2$ that is not a universal vertex in $V_2$ is adjacent to $x$. Finally, suppose that $V_{1,1} \not = \emptyset $. Then if any $v \in V_2$ is not adjacent to $x$, the shortest paths from $x$ to $V_{1,1}$ would have length three and pass through $V_{1,2}$, which is not allowed by Condition 3 of Definition~\ref{dfn:Mycielski partition}. Hence if $V_{1,1} \not = \emptyset $, then $x$ is a universal vertex. 
		
		Conversely, if all of the above conditions are satisfied, then the partition is an $\mathcal{MGP}$-partition. 
	\end{proof}
	Notice that if $V_2 = \emptyset $ in the second family described in Corollary~\ref{n+alpha-1 characterisation}, then we obtain a star. 
	
	\begin{figure}
		\centering
		\begin{tikzpicture}[x=0.2mm,y=-0.2mm,inner sep=0.2mm,scale=1.0,thick,vertex/.style={circle,draw,minimum size=10,fill=lightgray}]
			\node at (-300,-100) [vertex] (v1) {};
			]\node at (-250,-100) [vertex] (v2) {};
			\node at (-200,-100) [vertex] (v3) {};
			\node at (-150,-100) [vertex] (v4) {};
			
			\node at (-300,-50) [vertex] (u1) {};
			\node at (-250,-50) [vertex] (u2) {};
			\node at (-200,-50) [vertex] (u3) {};
			\node at (-150,-50) [vertex] (u4) {};
			
			\node at (-225,0) [vertex] (x) {};

			\node at (-75,-75) [vertex] (v'1) {};
			\node at (-25,-75) [vertex] (v'2) {};
			\node at (25,-75) [vertex] (v'3) {};
			\node at (75,-75) [vertex] (v'4) {};
			
			\node at (-75,-25) [vertex] (u'1) {};
			\node at (-25,-25) [vertex] (u'2) {};
			\node at (25,-25) [vertex] (u'3) {};
			\node at (75,-25) [vertex] (u'4) {};
			
			\node at (0,25) [vertex] (x') {};

			\node at (-275,50) [vertex] (l1) {};
			\node at (-225,50) [vertex] (l2) {};
			\node at (-175,50) [vertex] (l3) {};

			\node at (200,-75) [vertex] (z1) {};
			\node at (128.67,-23.18) [vertex] (z2) {};
			\node at (155.92,60.68) [vertex] (z3) {};
			\node at (244.08,60.68) [vertex] (z4) {};
			\node at (271.33,-23.18) [vertex] (z5) {};
			
			\node at (200,-150) [vertex] (y) {};
			
			\node at (388,-110) [vertex] (w1) {};
			\node at (311.92,-54.72) [vertex] (w2) {};
			\node at (340.98,34.72) [vertex] (w3) {};
			\node at (435.02,34.72) [vertex] (w4) {};
			\node at (464.08,-54.72) [vertex] (w5) {};
			\node at (388,-30) [vertex] (w) {};
			
			\path
			
			(w) edge (w1)
			(w) edge (w2)
			(w) edge (w3)
			(w) edge (w4)
			(w) edge (w5)
			
			(x') edge (u'1)
			(x') edge (u'2)
			(x') edge (u'3)
			(x') edge (u'4)
			
			(x') edge[bend right = 10] (v'2)
			(x') edge[bend left = 10] (v'3)
			
			(v'1) edge (v'2)
			(v'1) edge[bend left] (v'3)
			(v'1) edge[bend left] (v'4)
			(v'4) edge[bend right] (v'2)
			(v'4) edge (v'3)
			
			(u'1) edge (v'1)
			(u'2) edge (v'2)
			(u'3) edge (v'3)
			(u'4) edge (v'4)

			(y) edge (z1)
			(z1) edge (z2)
			(z1) edge (z3)
			(z1) edge (z4)
			(z1) edge (z5)
			(z2) edge (z3)
			(z2) edge (z4)
			(z2) edge (z5)
			(z3) edge (z4)
			(z3) edge (z5)
			(z4) edge (z5)
			
			(v1) edge (u1)
			(v2) edge (u2)
			(v3) edge (u3)
			(v4) edge (u4)
			
			(x) edge (u1)
			(x) edge (u2)
			(x) edge (u3)
			(x) edge (u4)
			
			(x) edge[bend left =20] (v1)
			(x) edge (v2)
			(x) edge (v3)
			(x) edge[bend right =20] (v4)
			
			(x) edge (l1)
			(x) edge (l2)
			(x) edge (l3)

			(v1) edge (v2)
			(v3) edge (v4)

			;
		\end{tikzpicture}
		\caption{Some graphs with $\gp(\mathcal{M}(G))=n+\alpha-1$.}
		\label{fig:n+alpha-1 graphs}
	\end{figure}

	\section{Regular graphs}\label{sec:reg graphs}
	
	It was indicated in Section~\ref{General bound} that the regular graphs provide a rich source of meagre graphs. In this section we show that there is a strong upper bound on $\gp (\mathcal{M}(G))$ for regular graphs $G$ and that for any $d$ there are finitely many abundant $d$-regular graphs.
	
	\begin{thm}\label{thm:regular bound}
		If $G$ is a $d$-regular graph with order $n \geq 3$, then \[ \gp (\mathcal{M}(G)) \leq n+\left \lfloor \frac{d-1}{2}+\frac{1}{d} \right \rfloor. \]
	\end{thm} 
	\begin{proof}
		Let $G$ be a $d$-regular graph, $S$ be a gp-set of $\mathcal{M}(G)$ that does not include $u^*$ and $(V_1,V_2,V_3,V_4)$ be the corresponding $\mathcal{MGP}$-partition. Suppose that $(V_1,V_2) = \emptyset $. Then all neighbours of vertices in $V_1$ lie in $V_4$. This implies that $n_4d \geq n_1d$, yielding $n_4 \geq n_1$, so that $\gp (\mathcal{M}(G)) = n$.
		
		Hence we may assume that $|(V_1,V_2)| = r > 0$. As $(V_1,V_2)$ is a matching, there are $n_1d-r$ edges in $(V_1,V_4)$, implying that $n_4 \geq n_1-\frac{r}{d}$. Thus
		\[ \gp (\mathcal{M}(G)) = n+n_1-n_4 \leq n+\frac{r}{d}.\]
		By Condition 2 of Definition~\ref{dfn:Mycielski partition}, if $u \in V_1$ is adjacent to a vertex $v \in V_2$, then all vertices of $(V_1\cup V_2)\setminus \{ u,v\} $ lie in $N^2(u)$. There are at least $2r$ vertices in $V_1\cup V_2$ and at most $d^2-d$ vertices in $N^2(u)$ and so $2r \leq d^2-d+2$. It follows that
		\[ \gp (\mathcal{M}(G)) \leq n+\frac{d-1}{2}+\frac{1}{d}. \]
		In particular, for $d \geq 3$ we have $\gp (\mathcal{M}(G)) \leq n+\frac{d-1}{2}$. 
	\end{proof}
	
	Theorem~\ref{thm:regular bound} suggests the following construction of abundant regular graphs.
	\begin{thm}\label{thm: n+1 construction}
		For all $d \geq 2$, there is an abundant $d$-regular graph $G(d)$ with order $n = 3d-1$ and $\gp (\mathcal{M}(G(d))) = n+1$. 
	\end{thm}
	\begin{proof}
		We form the graph $G(d)$ as follows. Take three sets $V_1,V_2,V_4$ of vertices, where $|V_1| = |V_2| = d$ and $|V_4| = d-1$. Form the complete bipartite graph $K_{d,d-1}$ with partite sets $V_1$ and $V_4$. Make $V_2$ into a clique $K_d$ by adding an edge between every pair of vertices in $V_2$. Finally join $V_1$ and $V_2$ by a matching of size $d$. The graph $G(5)$ is shown in Figure~\ref{fig:reg n+1}. As suggested by our labelling of the sets, $(V_1,V_2,\emptyset ,V_4)$ is an $\mathcal{MGP}$-partition of $G(d)$. As $\ip _4(G(d)) = \alpha (G(d)) = d$, these graphs are abundant.

		\begin{figure}
			\centering
			\begin{tikzpicture}[x=0.2mm,y=-0.2mm,inner sep=0.2mm,scale=0.8,thick,vertex/.style={circle,draw,minimum size=10,fill=lightgray}]
				\node at (380,272) [vertex] (v1) {};
				\node at (316.5,467.4) [vertex] (v2) {};
				\node at (482.7,346.6) [vertex] (v3) {};
				\node at (277.3,346.6) [vertex] (v4) {};
				\node at (443.5,467.4) [vertex] (v5) {};
				
				\node at (80,272) [vertex] (u1) {};
				\node at (16.5,467.4) [vertex] (u2) {};
				\node at (182.7,346.6) [vertex] (u3) {};
				\node at (-32.7,346.6) [vertex] (u4) {};
				\node at (143.5,467.4) [vertex] (u5) {};

				\node at (-150,272) [vertex] (w1) {};
				\node at (-150,337.13) [vertex] (w2) {};
				\node at (-150,402.26) [vertex] (w3) {};
				\node at (-150,467.4) [vertex] (w4) {};
				\path
				
				(w1) edge (u1)
				(w1) edge (u2)
				(w1) edge (u3)
				(w1) edge (u4)
				(w1) edge[bend left] (u5)
				
				(w2) edge (u1)
				(w2) edge (u2)
				(w2) edge[bend left] (u3)
				(w2) edge (u4)
				(w2) edge (u5)
				
				(w3) edge (u1)
				(w3) edge (u2)
				(w3) edge (u3)
				(w3) edge (u4)
				(w3) edge (u5)
				
				(w4) edge[bend right] (u1)
				(w4) edge (u2)
				(w4) edge (u3)
				(w4) edge (u4)
				(w4) edge[bend right] (u5)

				(u1) edge[bend left] (v1)
				(u2) edge[bend left] (v2)
				(u3) edge[bend left] (v3)
				(u4) edge[bend left] (v4)
				(u5) edge[bend left] (v5)
				
				(v1) edge (v2)
				(v1) edge (v3)
				(v1) edge (v4)
				(v1) edge (v5)
				(v2) edge (v3)
				(v2) edge (v4)
				(v2) edge (v5)
				(v3) edge (v4)
				(v3) edge (v5)
				(v4) edge (v5)
				
				;
			\end{tikzpicture}
			\caption{A 5-regular graph $G$ with $\gp (\mathcal{M}(G)) = n+1$}
			\label{fig:reg n+1}
		\end{figure} 	
	\end{proof}
	We now classify the abundant $2$- and $3$-regular graphs. It turns out that the graphs $G(2)$ and $G(3)$ are the only abundant regular graphs with degree three or less.	
	
	\begin{thm}\label{thm: cycles}
		For $n \geq 3$, the Mycielskian of the cycle $C_n$ has general position number
		
		$$\gp(\mathcal{M}(C_n)) =
		{
			\begin{cases}
				n+1,&\text{ if } n = 3 \text{ or } 5 \\	
				n, & \text{ otherwise.}
		\end{cases}}$$
	\end{thm}
	\begin{proof}
		The result for $n = 3$ follows from Corollary~\ref{cor:completes}. For $n \geq 4$ let $S$ be a gp-set of $\mathcal{M}(C_n)$ that does not include $u^*$ and consider the corresponding $\mathcal{MGP}$-partition $(V_1,V_2,V_3,V_4)$ of $V(C_n)$. Suppose that $\gp(\mathcal{M}(C_n)) \geq n+1$. By Theorem~\ref{thm:regular bound} we have $\gp(\mathcal{M}(C_n)) \leq n+1$, so we must have exactly $\gp (\mathcal{M}(C_n))=n+1$. As we have equality in the bound of Theorem~\ref{thm:regular bound}, it follows that the matching $(V_1,V_2)$ has size $r = 2$ and $n_4 = n_1-1$. 
		
		Let $u_1,u_2 \in V_1$ and $v_1,v_2 \in V_2$, where $u_1 \sim v_1$, $u_2 \sim v_2$ are the edges of $(V_1,V_2)$ . It follows from Theorem~\ref{thm:regular bound} that $\{ u_2,v_2\} = N^2(u_1)$. Since $u_2 \sim v_2$, the cycle has length five and all cycles with length $n = 4$ or $n \geq 6$ are meagre. However, as can be seen in Figure~\ref{fig:Grotzsch graph}, this argument does yield a gp-set of cardinality six in the Gr\"otzsch graph $\mathcal{M}(C_5)$. Notice that $C_5$ is isomorphic to $G(2)$ from Theorem~\ref{thm: n+1 construction}.	
	\end{proof}
	
	\begin{figure}
		\centering
		\begin{tikzpicture}[x=0.2mm,y=-0.2mm,inner sep=0.2mm,scale=0.8,thick,vertex/.style={circle,draw,minimum size=10,fill=lightgray}]
			\node at (380,200) [vertex,color=red] (v1) {};
			\node at (208.8,324.4) [vertex] (v2) {};
			\node at (274.2,525.6) [vertex] (v3) {};
			\node at (485.8,525.6) [vertex,color=red] (v4) {};
			\node at (551.2,324.4) [vertex] (v5) {};
			\node at (380,272) [vertex,color=red] (v6) {};
			\node at (316.5,467.4) [vertex,color=red] (v7) {};
			\node at (482.7,346.6) [vertex] (v8) {};
			\node at (277.3,346.6) [vertex,color=red] (v9) {};
			\node at (443.5,467.4) [vertex,color=red] (v10) {};
			\node at (380,380.01) [vertex] (x) {};
			\path
			(v1) edge (v2)
			(v1) edge (v5)
			(v2) edge (v3)
			(v3) edge (v4)
			(v4) edge (v5)
			
			(v6) edge (x)
			(v9) edge (x)
			(v7) edge (x)
			(v10) edge (x)
			(v8) edge (x)
			
			(v2) edge (v6)
			(v2) edge (v7)
			(v1) edge (v9)
			(v1) edge (v8)
			(v5) edge (v6)
			(v5) edge (v10)
			(v4) edge (v8)
			(v4) edge (v7)
			(v3) edge (v9)
			(v3) edge (v10)
			;
		\end{tikzpicture}
		\caption{A gp-set (in red) for $\mathcal{M}(C_5)$}
		\label{fig:Grotzsch graph}
	\end{figure}

	\begin{thm}\label{thm:cubic graphs}
		The graph $G(3)$ is the unique non-complete abundant cubic graph.  
	\end{thm}
	\begin{proof}
		Let $G$ be a non-complete cubic graph with order $n$ and $\gp(\mathcal{M}(G)) \geq n+1$. The proof of Theorem~\ref{thm:regular bound} shows that $\gp (\mathcal{M}(G)) = n+1$ and the matching $(V_1,V_2)$ has size $r = 3$ or $4$. 
		
		Suppose firstly that $r = 4$ and let $u_i \sim v_i$ for $i = 1,2,3,4$ be the edges of $(V_1,V_2)$. By Condition 2 of Definition~\ref{dfn:Mycielski partition} we have $\{ u_2,u_3,u_4,v_2,v_3,v_4\} \subseteq N^2(u_1)$. As there are at most six vertices in $N^2(u_1)$, it follows that $\{ u_2,u_3,u_4,v_2,v_3,v_4\} = N^2(u_1)$, $n_1 = n_2 = 4$ and $n_4 = 3$. Again by Condition 2 of Definition~\ref{dfn:Mycielski partition} any vertex $w$ of $V_3$ lies in $N^3(u)$. However, a shortest $u_1,w$-path would pass through $V_1 \cup V_2$, which is forbidden by Condition 3 of Definition~\ref{dfn:Mycielski partition}. Thus $V_3 = \emptyset $. However, this implies that $G$ has order $n = n_1+n_2+n_4 = 11$, whereas any cubic graph has even order, a contradiction.    
		
		Hence we can assume that $r = 3$. Write $V_1 = \{ u_1,u_2,\dots, u_{n_1}\} $ and $V_2 = \{ v_1,v_2,\dots ,v_{n_2}\} $, where $u_i \sim v_i$ are the edges of the matching $(V_1,V_2)$ for $i = 1,2,3$. As above, $(V_1 \cup V_2) \setminus \{ u_1,v_1\} \subseteq N^2(u_1)$, so have $n_1+n_2 \leq 8$. Since $n_1,n_2 \geq r = 3$, $n_1$ lies in the range $3 \leq n_1 \leq 5$. If $n_1+n_2 = 8$, then $(V_1 \cup V_2) \setminus \{ u_1,v_1\} = N^2(u_1)$ and as before $V_3$ would be empty.
		
		Suppose that $n_1 = 5$. Then as $3 \leq n_2 \leq 8-n_1 = 3$, we have $n_2 = 3$, $n_3 = 0$ and $n_4 = 4$, giving $n = 12$. For $i = 1,2,3$, we have $V_1 \cup V_2 = \{ u_i,v_i\} \cup N^2(u_i)$. As each vertex of $V_2$ has just one edge to $V_1$, both vertices of $N_G(v_i) \cap N^2(u_i)$ must lie in $V_2$. Hence $V_2$ induces a triangle in $G$. Thus there is no path of length at most three from $v_1$ to $u_4$ via $\{ v_2,v_3\} $, so the path $v_1,u_1,w,u_4$, where $w$ is a neighbour of $u_1$ in $V_4$, is a shortest path of the form forbidden by Condition 3 of Definition~\ref{dfn:Mycielski partition}, a contradiction.
		
		Now suppose that $n_1 = 4$, so that $n_4 = 3$ and $n_2 \in \{ 3,4\} $. If $n_2 = 4$, then $n_3 = 0$, giving $n = 11$, whereas $G$ must have even order. Thus $n_2 = 3$. By Condition 2 of Definition~\ref{dfn:Mycielski partition} all vertices of $n_3$ lie at distance at most three from $u_i$ for $i = 1,2,3$. There is just one vertex $x_i$ of $N^2(u_i)$ that does not belong to $V_1 \cup V_2$ for $i = 1,2,3$ and $x_i$ is the only vertex of $N^2(u_i)$ that can have neighbours in $V_3$ by Condition 3 of Definition~\ref{dfn:Mycielski partition}. Thus $n_3 \leq 3$ and, by the parity of $n$, either $n_3 = 0$ or $n_3 = 2$. Write $N_G(x_i) \cap N^3(u_i) = \{ y_1,y_2\} $. If $x_i \in V_3$, then by Condition 3 of Definition~\ref{dfn:Mycielski partition} neither $y_1$ nor $y_2$ is in $V_3$, contradicting the parity requirement. Thus $x_i \in V_4$ and for each of $u_1,u_2,u_3$ two vertices of $V_4$ lie in $N_G(u_i)$ and one vertex of $V_4$ lies in $N^2(u_i)$. Write $V_4 = \{ x_1,w_1,w_2\} $, where $N_G(u_1) = \{ w_1,w_2,v_1\} $.  
		
		If $n_3 = 2$, then, by Condition 3 of Definition~\ref{dfn:Mycielski partition}, $\{ x_1\}  \cup (V_1 \setminus \{ u_1\}) \cap N_G(v_1) = \emptyset $. Assume without loss of generality that $x_1 \sim w_2$. Then at least one vertex of $\{ u_2,u_3\} $, say $u_2$, lies in $N_G(w_1)$,and we obtain the configuration in Figure~\ref{fig:Moore tree cubic} (although note that the positions of $u_3$ and $u_4$ might be swapped). But now $u_2$ has at most one neighbour in $V_4$, contradicting our previous deduction. Therefore $n_3 = 0$.  
		
		\begin{figure}\centering
			\begin{tikzpicture}[x=0.2mm,y=-0.2mm,inner sep=0.1mm,scale=4,
				thick,vertex/.style={circle,draw,minimum size=16,font=\tiny,fill=white},edge label/.style={fill=white}]
				
				\node at (0,0) [vertex] (v1) {$u_1$};
				
				\node at (-50,20) [vertex] (v2) {$w_1$};
				\node at (0,20) [vertex] (v3) {$w_2$};
				\node at (50,20) [vertex] (v4) {$v_1$};
				
				\node at (-60,40) [vertex] (v5) {$u_2$};
				\node at (-40,40) [vertex] (v6) {$u_4$};
				\node at (-10,40) [vertex] (v7) {$w_3$};
				\node at (10,40) [vertex] (v8) {$u_3$};
				\node at (40,40) [vertex] (v9) {};
				\node at (60,40) [vertex] (v10) {};
				
				\node at (-20,60) [vertex] (v11) {$x_1$};
				\node at (0,60) [vertex] (v12) {$x_2$};
				\path
				(v1) edge (v2)
				(v1) edge (v3)
				(v1) edge (v4)
				(v2) edge (v5)
				(v2) edge (v6)
				(v3) edge (v7)
				(v3) edge (v8)
				(v4) edge (v9)
				(v4) edge (v10)
				(v7) edge (v11)
				(v7) edge (v12)
				;	
				
			\end{tikzpicture}
			\caption{Configuration for $n_1 = 4, n_2 = 3, n_3 = 2, n_4 = 3$.}
			\label{fig:Moore tree cubic}
		\end{figure}
		
		Thus suppose that $n_1 = 4$, $n_2 = n_4 = 3$ and $n_3 = 0$. Write $V_4 = \{ w_1,w_2,w_3\} $. Then each of $u_1,u_2,u_3$ has one edge to $V_2$ and all other edges incident with $V_1$ are to $V_4$. In particular, $N_G(u_4) \subset V_4$. Thus $|(V_1,V_4)| = 9$ and every vertex of $V_4$ has neighbourhood entirely contained in $V_1$. Hence each vertex of $V_2$ must have one neighbour in $V_1$ and two neighbours in $V_2$, so $V_2$ induces a triangle. This implies that $G$ is isomorphic to the graph in Figure~\ref{fig:cubic graph 2}. However, this graph contains a shortest path $v_1,u_1,w_3,u_4$, contradicting Condition 3 in Definition~\ref{dfn:Mycielski partition}.
		
		\begin{figure}\centering
			\begin{tikzpicture}[x=0.2mm,y=-0.2mm,inner sep=0.1mm,scale=0.8,
				thick,vertex/.style={circle,draw,minimum size=16,font=\tiny,fill=white},edge label/.style={fill=white}]

				\node at (0,0) [vertex] (u4) {$u_4$};
				\node at (100,0) [vertex] (u1) {$u_1$};
				\node at (50,86.60) [vertex] (w1) {$w_1$};
				\node at (-50,86.60) [vertex] (u2) {$u_2$};
				\node at (-100,0) [vertex] (w2) {$w_2$};
				\node at (-50,-86.60) [vertex] (u3) {$u_3$};
				\node at (50,-86.60) [vertex] (w3) {$w_3$};
				
				\node at (200,0) [vertex] (v1) {$v_1$};	
				\node at (-100,173.21) [vertex] (v2) {$v_2$};	
				\node at (-100,-173.21) [vertex] (v3) {$v_3$};	
				\path
				(v1) edge[bend left] (v2)
				(v2) edge[bend left] (v3)
				(v3) edge[bend left] (v1)
				
				(v1) edge (u1)
				(v2) edge (u2)
				(v3) edge (u3)
				
				(u4) edge (w1)
				(u4) edge (w2)
				(u4) edge (w3)
				(u1) edge (w1)
				(w1) edge (u2)
				(u2) edge (w2)
				(w2) edge (u3)
				(u3) edge (w3)
				(w3) edge (u1)
				;	
				
			\end{tikzpicture}
			\caption{Cubic graph for $n_1 = 4, n_2 = 3, n_3 = 0, n_4 = 3$.}
			\label{fig:cubic graph 2}
		\end{figure}

		Finally, let $n_1 = 3, n_4 = 2$. Each of $u_1,u_2,u_3$ has one edge to $V_2$ and two edges to $V_4$, making six edges from $V_1$ to $V_4$. Hence both vertices of $V_4$ have neighbourhood contained in $V_1$. For each of $u_1,u_2,u_3$ to reach each vertex of $V_2$ in distance at most two, we see that $n_2 = 3$ and $V_2$ induces a triangle in $G$, so that $V_3 = \emptyset $. Hence $G$ is isomorphic to $G(3)$, completing the proof. \end{proof}

	\begin{thm}\label{thm:largeregulargraphsaremeagre}
		Any $d$-regular graph with order $\geq d^3-2d^2+2d+2$ is meagre. 
	\end{thm}
	\begin{proof}
		Let $G$ be an abundant $d$-regular graph. We can assume that $r = |(V_1,V_2)| \geq d > 0$. Let $u_1 \sim v_1$ be an edge of the matching $(V_1,V_2)$, where $u_1 \in V_1$, $v_1 \in V_2$. Consider a breadth-first search tree $T$ of depth three rooted at $u_1$. By Condition 2 in Definition~\ref{dfn:Mycielski partition}, all vertices of $(V_1 \cup V_2)\setminus \{ u_1,v_1\} $ are contained in $N^2(u)$ and $V_3$ is contained in $N^2(u) \cup N^3(u)$. There are at most $d^2-d$ vertices at distance two from $u_1$ and at most $d(d-1)^2$ vertices in $N^3(u_1)$. 
		
		Moreover there are at least $2d-2$ vertices of $V_1 \cup V_2$ in $N^2(u)$ and no such vertex can have descendants at distance three from $u_1$ in $V_3$ by Condition 3 in Definition~\ref{dfn:Mycielski partition}. Therefore we see that 
		\[ |V_1\cup V_2 \cup V_3| \leq 2+d^2-d+[d^2-d-(2d-2)](d-1) = d^3-3d^2+4d.\]
		Since we are assuming that $G$ is abundant, we can bound $n_4$ as follows:
		\[ n_4 \leq n_1-1 \leq 1+(d^2-d)-(d-1)-1 = d^2-2d+1. \]
		In total then the order of $G$ is bounded above by $d^3-2d^2+2d+1$.
	\end{proof}
	This shows that for any value of $d$ there are only a finite number of abundant $d$-regular graphs.	This raises the question whether we can strengthen this result by making the same conclusion for graphs with given maximum degree. The answer is no, as evidenced by the following construction. Let the maximum degree be $\Delta \geq 4$. Let $P$ be a path $u_0,u_1,\dots ,u_r$, where $r \geq 3$, and take the Cartesian product $P \cp K_2$ to form a ladder graph. If we identify the vertices of $K_2$ with $\{ 0,1\} $, then we can write the vertices of the ladder as $(u_i,j)$ for $0 \leq i \leq r$ and $j = 0,1$. Finally add $\Delta -3$ leaves to each vertex $(u_i,0)$ for $1 \leq i \leq r$ and $\Delta -4 $ leaves to $(u_0,0)$ and $(u_r,0)$. By taking the set of leaves to be $V_1$, $V_2 = V(P) \times \{ 1\} $ and $V_4 = V(P) \times \{ 0\} $ we obtain an $\mathcal{MGP}$-partition that shows that the graph is abundant. By increasing $r$ we see that the number of vertices in this family is unbounded for fixed $\Delta $.
	
	For degrees two and three, we have shown that our construction $G(d)$ is the unique abundant graph. Also, for both of these graphs the general position number of the Mycielskian is just one more than the order of the graph. This suggests the following problem.

	\begin{prb}
		For given $d \geq 4$, what is the largest value of $\gp (\mathcal{M}(G))-n $ over all $d$-regular graphs? Is $G(d)$ the unique abundant graph? 
	\end{prb}
	
	\section{The Mycielskian of trees and graphs with large girth}\label{Trees}
	
	We now investigate the general position number of the Mycielskian of trees and, as a by-product, give an upper bound on $\gp (\mathcal{M}(G))$ for graphs $G$ without short cycles. Let $T$ be any tree with order $n \geq 3$ and $L$ be the set of all $\ell(T)$ leaves of $T$. Recall that $T$ is \emph{abundant} if $\gp (\mathcal{M}(T)) > \max\{ n,2\ip _4(T)\} $ and otherwise is \emph{meagre}. We present an exact expression for $\gp (\mathcal{M}(T))$ for a wide class of trees.
	
	A vertex of $T$ is a \emph{support vertex} if it is adjacent to a leaf of $T$. We denote the set of support vertices of $T$ by $\Sigma = \{ y_1,y_2,\dots ,y_s\} $ and the number of support vertices of $T$ by $s = |\Sigma |$, whilst $Z = \{ z_1,z_2,\dots ,z_t\} $ is the set of vertices in $V(T)\setminus (L \cup \Sigma )$, so that $n = \ell +s+t$. We associate with each support vertex $y_i$ the set $L_i$ of leaves to which it is adjacent in $T$. Let $D_T = \min\{ d_T(l,l^\prime ): l \in L_i, l^\prime  \in L_j, i \not = j\} $, where $d_T$ represents the distance in $T$. If $s = 1$, then $T$ is a star and by Corollary~\ref{n+alpha-1 characterisation} we have $\gp (\mathcal{M}(T)) = 2\ell(T) = 2\ip _4(T)$, so that $T$ is meagre. Hence we can assume that $s \geq 2$ and $D_T$ is defined.  
	
	We begin with a lower bound on $\gp (\mathcal{M}(T))$ for trees with support vertices at distance at least three from each other and then proceed to prove that this bound is sharp. 
	
	\begin{lem}\label{lower bound n+l-w}
		If $T$ has $\ell $ leaves and $s$ support vertices and also $D_T \geq 5$, then the general position number of the Mycielskian $\mathcal{M}(T)$ satisfies
		\[ \gp(\mathcal{M}(T)) \geq n+\ell-s.\] 
	\end{lem}
	\begin{proof}
		Recall that $L$ is the set of leaves of $T$, $\Sigma $ the set of support vertices and $Z$ is $V(T) \setminus (L \cup \Sigma )$. $(L,Z,\emptyset ,\Sigma )$ is an $\mathcal{MGP}$-partition, so it follows from Corollary~\ref{cor:characterisation} that $S = L \cup L^\prime  \cup Z^\prime $ is in general position in $\mathcal{M}(T)$. This implies that $\gp(\mathcal{M}(T)) \geq 2\ell+t = n+\ell-s$.
	\end{proof}
	In the following, let $T$ be an abundant tree and $S$ be a gp-set of $\mathcal{M}(T)$, which we can assume by Lemma~\ref{lem:completes} does not contain $u^*$. We now present an upper bound on $\gp (\mathcal{M}(G))$ for graphs with girth at least six. Recall by our convention that the girth of a tree is $\infty $, so the following result applies in particular to trees.

	\begin{thm}\label{matching upper bound}
		Let $G$ be a graph with order $n \geq 4$, girth $\g(G) \geq 6$ and matching number $\nu (G)$. Then
		\[ \gp(\mathcal{M}(G)) \leq 2n-2\nu (G).\]
		Also, if $G$ satisfies $\gp(\mathcal{M}(G)) > n$ and $\pi = (V_1,V_2,V_3,V_4)$ is an $\mathcal{MGP}$-partition corresponding to a gp-set $S$ of $\mathcal{M}(G)$ that does not contain $u^*$, then $(V_1,V_2) = (V_2,V_3) = \emptyset $.
	\end{thm}
	\begin{proof}
		Let $G$, $S$ and $\pi $ be as described. Suppose that the matching $(V_1,V_2)$ contains an edge $u \sim v$, where $u \in V_1$ and $v \in V_2$. By Condition 2 of Definition~\ref{dfn:Mycielski partition}, all vertices of $(V_1 \cup V_2) \setminus \{ u,v\} $ are contained in $N^2(u)$. 
		
		Suppose that $n_1 \geq 2$ and let $w \in V_1\setminus \{ u\} $. By Conditions 1 and 2 of Definition~\ref{dfn:Mycielski partition}, $w \not \in N_G(u) \cup N_G(v)$, so $G$ contains a path $w,x,u,v$, where $x \in N_G(u)\setminus \{ v\} $. By Condition 3 of Definition~\ref{dfn:Mycielski partition}, this cannot be a shortest path, so there must be a path of length two from $w$ to $v$, implying the existence of a cycle of length at most five, a contradiction. 
		
		Thus in this case $n_1 = 1$. Thus if $\gp (\mathcal{M}(G)) > n$, then $n_4 = 0$. As $u$ has no neighbours in $V_3 \cup V_4$ and just one neighbour in $V_2$, $u$ is a leaf with support vertex $v$. Thus any of the paths from $u$ to $V_3$ guaranteed by Condition 2 must pass through $v_2$, violating Condition 3. Thus $V_3 = \emptyset $. As $G$ is triangle-free, it follows that $G$ is a star. However, by Corollary~\ref{bound}, the star $S_n$ has $\gp(\mathcal{M}(S_n)) = 2n-2$, whereas in the partition described above all but one vertex of $G$ lies in $V_2$. Since $S$ is assumed to be a gp-set, this implies that $\gp (\mathcal{M}(G)) = n+1 = 2n-2$, yielding $n = 3$. We conclude that $(V_1,V_2) = \emptyset $.
		
		Suppose now that there is an edge $v \sim w$ in $G$, where $v \in V_2$ and $w \in V_3$. By Condition 2 of Definition~\ref{dfn:Mycielski partition}, all vertices of $V_1$ lie in $N^2(w)$. By Condition 2, $N_G(v) \cap V_1 = \emptyset $. Thus, for any $u \in V_1$, there is a path $u,x,w,v$ for some $x \in V(G)$, and by Condition 3 this cannot be a shortest path. Again this implies the existence of a cycle with length at most five if $V_1 \not = \emptyset $. 
		
		If $\gp(\mathcal{M}(G)) = n$, then the claimed bound certainly holds, so assume that $\gp(\mathcal{M}(G)) > n$. Let $M$ be a maximum matching of $G$. If $x \sim y$ is any edge in $M$, then it follows from $(V_1,V_2) = (V_1,V_3) = \emptyset $ that one of $x,y$ is in $V_4$ and the vertices $\{ x,y,x^\prime ,y^\prime \} $ of $\mathcal{M}(G)$ can contribute at most two to $\gp(\mathcal{M}(G))$. Summing over all edges of $M$, we see that the vertices saturated by $M$ and their $\mathcal{M}$-twins contribute at most $2\nu(G)$ to $\gp(\mathcal{M}(G))$. Assuming that every other vertex lies in $V_1$, we obtain the upper bound $\gp(\mathcal{M}(G)) \leq 2\nu(G) +2(n-2\nu(G)) = 2n-2\nu(G)$.  
	\end{proof}
	The star $S_n$ shows that the bound of Theorem~\ref{matching upper bound} is tight.
	
	\begin{coro}\label{perfect matching}
		If $G$ is a graph with order $n \geq 4$ and girth $\g(G) \geq 6$ that has a perfect matching, then $\gp(\mathcal{M}(G)) = n$.
	\end{coro}
	
	\begin{coro}\label{(V_1,V_2) and (V_2,V_3) empty}
		If $T$ is a tree with order $n \geq 3$ and matching number $\nu (T)$, then $\gp(\mathcal{M}(T)) \leq 2n-2\nu(T)$. If $\gp (\mathcal{M}(T)) > n$, then $(V_1,V_2)$ and $(V_2,V_3)$ are empty. 
	\end{coro}

	\begin{lem}\label{lem:meagre tree}
		Let $T$ be any tree with order $n \geq 3$ and $\ell = s$, i.e. each support vertex of $T$ is adjacent to exactly one leaf. Then $\gp(\mathcal{M}(T)) = n$.
	\end{lem}
	\begin{proof}
		Let $T$ be as described and assume for a contradiction that $\gp (\mathcal{M}(T)) > n$. By Lemma~\ref{cor:partition} and Corollary~\ref{(V_1,V_2) and (V_2,V_3) empty}, every vertex of $V_1$ has neighbourhood lying in $V_4$. Consider the subgraph $\langle V_1 \cup V_4\rangle $ of $T$. As $n_1 > n_4$, there is no matching in this subgraph that saturates $V_1$ and we can choose a smallest subset $\widetilde{V_1} \subseteq V_1$ such that $\widetilde{V_4} = N_T(\widetilde{V_1})$ satisfies $|\widetilde{V_4}| < |\widetilde{V_1}|$. In particular, the subtree $\langle \widetilde{V_1} \cup \widetilde{V_4}\rangle $ is connected. 
		
		If any $v \in \widetilde{V_4}$ has just one neighbour $u$ in $\widetilde{V_1}$, then we could delete $u$ and $v$ from $\widetilde{V_1}$ and $\widetilde{V_4}$ to obtain a smaller pair with $ |\widetilde{V_1} \setminus \{ u\} | < |\widetilde{V_4} \setminus \{ v\} | $, contradicting minimality of $\widetilde{V_1}$. Thus let $u_1,u_2 \in \widetilde{V_1}$, $w_1 \in \widetilde{V_4}$, where $u_1,u_2 \in N_T(w_1)$. By our assumption $\ell = s$, not both of $u_1,u_2$ are leaves, so there is a $w_2 \not = w_1$ with $u_2 \sim w_2$. Now by the preceding argument $w_2$ must have a neighbour $u_3 \not = u_2$ in $\widetilde{V_1}$. To avoid cycles, we also have $u_3 \not = u_1$. Therefore we have constructed a path $u_1,w_1,u_2,w_2,u_3$, where $u_1,u_2,u_3 \in V_1$. This path is a geodesic, and so violates Condition 3 of Definition~\ref{dfn:Mycielski partition}. This contradiction establishes the result.     
	\end{proof}
	Lemma~\ref{lem:meagre tree} shows in particular that paths $P_n$ with $n \geq 4$ are meagre (and $P_3$ is also easily seen to be meagre).
	\begin{coro}\label{cor:treeupperbound}
		If $T$ is any tree with order $n \geq 3$, leaf number $\ell $ and $s$ support vertices, then $\gp (\mathcal{M}(T)) \leq n+\ell -s$.
	\end{coro}
	\begin{proof}
		We prove the result by induction on $\ell -s$. If $\ell = s$, then Lemma~\ref{lem:meagre tree} shows that the result is true. Suppose that $T$ satisfies $\gp (\mathcal{M}(T)) > n$ and construct the sets $\widetilde{V_1}$ and $\widetilde{V_4}$ as in the proof of Lemma~\ref{lem:meagre tree}. As in the previous lemma, we can assume that there is a vertex $w_1 \in \widetilde{V_4}$ that is adjacent to at least two leaves $u_1,u_2$ in $\widetilde{V_1}$. Set $\widetilde{T} = T \setminus \{u_1\} $. By induction, $\gp (\mathcal{M}(\widetilde{T})) \leq (n-1)+(\ell-1)+s$. Adding back the leaf $u_1$, the vertices $u_1,u_1^\prime $ can contribute at most two to $\gp (\mathcal{M}(T))$, so $\gp (\mathcal{M}(T)) \leq (n-2+\ell -s)+2 = n+\ell -s$. 
	\end{proof}
	Combining the upper bound from Corollary~\ref{cor:treeupperbound} with the lower bound from Lemma~\ref{lower bound n+l-w} we obtain the exact value of $\gp (\mathcal{M}(T))$ when $D_T \geq 5$.
	\begin{coro}\label{cor:main tree theorem}
		If $T$ is any tree with order $n \geq 3$, $D_T \geq 5$,  leaf number $\ell $ and $s$ support vertices, then $\gp (\mathcal{M}(T)) = n+\ell -s$. 
	\end{coro}

	We show finally that $\gp (\mathcal{M}(T)) = n+\ell -s$ also holds for some trees with $D_T \leq 4$. A \emph{caterpillar} is a tree formed by attaching leaves to the vertices of a path (called the \emph{central path}). 
	
	\begin{coro}\label{the:all leaves and support}
		If $T$ is a tree in which every vertex is either a leaf or a support vertex, then $\gp (\mathcal{M}(T)) = 2\ell (T)$ and $T$ is meagre.
	\end{coro}
	\begin{proof}
		If every vertex of $T$ is either a leaf, or a support vertex, then $n = \ell +s$ and the upper bound in Corollary~\ref{cor:treeupperbound} gives $\gp (\mathcal{M}(T)) \leq (\ell +s)+\ell -s = 2\ell $. It is shown in~\cite{ullas-2016} that the set of leaves $L$ is a general position set of $T$ and hence an independent 4-position set. Therefore the lower bound $2\ell $ from Corollary~\ref{thm:lower bound} matches the upper bound $n+\ell -s$ and $T$ is meagre.  
	\end{proof}
	Corollary~\ref{the:all leaves and support} shows in particular that all caterpillars with a leaf adjacent to each vertex of the central path is meagre. Equality in the bound $\gp (\mathcal{M}(T)) \leq n+\ell -s$ does not always hold, as can be easily seen by attaching several leaves to the endpoints of $P_3$.

	\section{Conclusion}
	
	We conclude with two promising directions for further research. The first obvious question is the complexity of finding the gp-number of Mycielskians. The proof of Lemma~\ref{lem:completes} shows that any general position sets containing $u^*$ and having cardinality at least $n+1$ can be found easily in polynomial time. Therefore, given the duality between general position sets of $\mathcal{M}(G)$ and $\mathcal{MGP}$-partitions of $V(G)$, the complexity of finding $\gp (\mathcal{M}(G))$ is the same as the complexity of finding the largest value of $n_1-n_4$ over all $\mathcal{MGP}$-partitions $(V_1,V_2,V_3,V_4)$ of $G$.
	
	\begin{prb}
		What is the complexity of finding the general position number of the Mycielskian $\mathcal{M}(G)$ of an arbitrary graph $G$?
	\end{prb}
	
	Secondly, it was mentioned in Section~\ref{sec:intro} that the general position number is one amongst a variety of `position-type' parameters. We suggest that it would be of interest to find the values of some of these other position numbers on Mycielskians of graphs, in particular the following three numbers. 
	
	\begin{prb}
		What are the mutual visibility, lower general position and monophonic position numbers of Mycielskians of graphs?
	\end{prb}

	\section*{Statements and Declarations}
	First of all, we thank Sudev Naduvath for suggesting this topic. The third author was supported by EPSRC grant EP/W522338/1 and London Mathematical Society grant ECF-2021-27. The authors have no relevant financial or non-financial interests to disclose. There are no datasets associated with the article. The authors thank the two reviewers for their helpful comments and Grahame Erskine for useful discussion.

\end{document}